\documentclass[12pt]{article}
\usepackage[margin=1.1in]{geometry}
\usepackage{graphicx}
\usepackage[labelfont=bf]{caption}
\usepackage{amsmath,amssymb,amsthm,bm,xfrac}
\usepackage{macros}
\usepackage{mdframed}
\usepackage{hyperref}
\usepackage{cancel}
\usepackage[capitalize]{cleveref}
\usepackage[dvipsnames]{xcolor}
\usepackage{subcaption}
\usepackage[sort,numbers]{natbib}

\newcommand{\ceil}[1]{\left\lceil {#1} \right\rceil}

\newcommand{\radius}{\varphi}

\newcommand{\constb}{C_{(b)}}
\newcommand{\bundle}{\mathtt{PolyakBundle}}
\newcommand{\polyak}{\mathtt{PolyakSGM}}
\newcommand{\superpolyak}{\mathtt{SuperPolyak}}

\newcommand{\est}{\mathsf{est}}

\newcommand{\constsuper}{C_{\mathsf{s}}}

\newcommand{\outind}{k}
\newcommand{\inner}{i}

\newcommand{\multo}{\rightrightarrows}

\newcommand{\subg}{g}
\newcommand{\jac}{G}

\newcommand{\closedball}{\bar{B}}

\newtheorem{theorem}{Theorem}[section]
\newtheorem{example}{Example}[section]
\newtheorem{lemma}[theorem]{Lemma}
\newtheorem{corollary}[theorem]{Corollary}
\Crefname{assumption}{Assumption}{Assumption}
\Crefname{corollary}{Corollary}{Corollaries}
\Crefname{theorem}{Theorem}{Theorems}
\Crefname{lemma}{Lemma}{Lemmas}
\Crefname{claim}{Claim}{Claims}

\theoremstyle{remark}
\newtheorem{remark}{Remark}
\theoremstyle{definition}

\Crefname{fact}{Fact}{Facts}

\usepackage{algorithm}
\usepackage{algpseudocode}
\allowdisplaybreaks

\begin{document}

	\title{A superlinearly convergent subgradient method for sharp semismooth problems}

	\author{
    Vasileios Charisopoulos\thanks{School of Operations Research and Information Engineering, Cornell University, Ithaca, NY 14853, USA;
    \texttt{people.orie.cornell.edu/vc333/}} \and
    Damek Davis\thanks{School of Operations Research and Information Engineering, Cornell University, Ithaca, NY 14850, USA;
    \texttt{people.orie.cornell.edu/dsd95/}. Research of Davis supported by an Alfred P. Sloan research fellowship and NSF DMS award 2047637.}}
\maketitle

\begin{abstract}~
Subgradient methods comprise a fundamental class of nonsmooth optimization algorithms.
Classical results show that certain subgradient methods converge sublinearly for general Lipschitz convex functions and converge linearly for convex functions that grow sharply away from solutions.
Recent work has moreover extended these results to certain nonconvex problems.
In this work we seek to improve the complexity of these algorithms, asking: is it possible to design a superlinearly convergent subgradient method?
We provide a positive answer to this question for a broad class of sharp  semismooth functions.
\end{abstract}

\section{Introduction}

Subgradient methods are a popular class of nonsmooth optimization algorithms for minimizing locally Lipschitz functions $f \colon \RR^d \rightarrow \RR$:
$$
\minimize_{x \in \RR^d}\;  f(x).
$$
Given an initial iterate $x_0 \in \RR^d$, the basic method repeats
$$
x_{k+1} = x_k - \alpha_k v_k \qquad \text{for $v_k \in \partial f(x_k)$},
$$
where $\{\alpha_k\}$ is a control sequence and $\partial f(x)$ denotes the \emph{Clarke subdifferential} at a point $x \in \RR^d$, comprised of limiting convex combinations of gradients at nearby points~\cite{RW98}.
While the method originated over fifty years ago in convex optimization~\cite{Goffin77,erem_subgrad,Polyak69,shor_rate_subgrad,subgrad_surv} (with later extensions to nonconvex problems~\cite{Nurminskii1973,Nurminskii1974,norkin85,ermol1998stochastic,davis2020stochastic}), it has recently become a popular and successful technique both in modern deep learning problems (e.g., in Google's Tensorflow~\cite{Tensorflow16}) and in robust low-rank matrix estimation problems~\cite{CCD+21}.
For the latter problem class, recent work has highlighted the prevalence and benefits of the so-called \emph{sharp growth} property, which stipulates that $f$ grows at least linearly away from its minimizers:
$$
f(x) - \inf f \geq \mu \cdot \dist(x, \cX_\ast),
$$
where $\cX_\ast = \argmin f$.
For convex problems (and more generally \emph{weakly convex} problems),
this classical regularity condition leads to local linear convergence provided the 
sequence $\{\alpha_k\}$ is chosen appropriately (see also \cite{Polyak69,Goffin77,shor_rate_subgrad,erem_subgrad,stagg,rsg,jstone,DDMP18}).
While linear convergence is desirable, we ask:
\begin{quote}
Is it possible to design a locally superlinearly convergent subgradient method?
\end{quote}
In this paper, we design such a method for a wide class of \emph{sharp} and \emph{semismooth} problems.

Setting the stage, assume for simplicity that $f$ has a unique minimizer $\bar x$ and optimal value $0$. The starting point for our method is the classical subgradient method with Polyak stepsize, which iterates
$$
x_{k+1} = x_k - \frac{f(x_k)}{\|v_k\|^2}v_k, \qquad \text{where $v_k \in \partial f(x_k)$.}
$$
This method converges linearly for sharp convex~\cite{Polyak69} and weakly
convex~\cite{DDMP18} problems and admits the following reformulation:
\begin{align}\label{eq:polyakreform}
x_{k+1} = \argmin_{x \in \RR^d} \|x - x_k\|^2 \qquad \text{subject to: $f(x_k) + \dotp{v_k, x - x_k} \leq 0$.}
\end{align}
Seeking to improve the linear convergence of~\eqref{eq:polyakreform}, a natural strategy proposed in Polyak's original work~\cite{Polyak69} is to augment the constraint~\eqref{eq:polyakreform} with a collection $\{(y_{i}, v_{i})\}_{i= 1}^n$ of points $y_{i}$ and subgradients $v_{i} \in \partial f(y_{i})$, resulting in the update:
\begin{align}\label{eq:polyakreform2}
x_{k+1} = \argmin_{x \in \RR^d} \|x - x_k\|^2 \qquad \text{subject to: $\begin{bmatrix} f(y_{i}) + \dotp{v_{i}, x - y_{i}}\end{bmatrix}_{i=1}^n \leq 0$},
\end{align}
The work~\cite{Polyak69} suggests choosing $y_i$ among $\{x_j\}_{j \leq k}$ and shows that the iterates $x_k$ converge linearly for Lipschitz convex functions (similar to~\eqref{eq:polyakreform}). 
While there is no theoretical convergence rate improvement, the work~\cite{Polyak69} suggests the method~\eqref{eq:polyakreform2} improves upon~\eqref{eq:polyakreform} numerically, though the per-iteration cost may grow substantially if $n$ is large. 

A strategy akin to~\eqref{eq:polyakreform2} also appears in the literature on so-called  \emph{bundle methods}~\cite{lemarechal1975extension,wolfe1975method}. Instead of aggregating inequalities as in~\cite{Polyak69}, these methods build piecewise linear models of the objective function and output the proximal point of the models. If the proximal point sufficiently decreases the objective, the algorithm takes a ``serious step." Otherwise, the algorithm takes a ``null step," which consists of using subgradient information to improve the model. Bundle methods often perform well in practice and their convergence/complexity theory is understood in several settings~\cite{kiwiel1985linearization,kiwiel2000efficiency,oliveira2014bundle,du2017rate,emiel2010incremental,sagastizabal2012divide,hare2010redistributed,liang2021proximal}. Most relevantly for this work, on sharp convex functions, variants of the bundle method converge superlinearly relative to the number of serious steps~\cite{vuoriginal} and converge linearly relative to both serious and null steps~\cite{diaz2021optimal}.

In this work, we study a slight variant of the update~\eqref{eq:polyakreform2}, where the ``$\leq$" is replaced by an equality and the points $y_i$ are chosen iteratively.
This variant is motivated by our second assumption -- semismoothness. In short, semismoothness ensures that $\bar x$ is nearly feasible for the equation $f(x_k) + \dotp{v, x - x_k} = 0$ when $x_k$ is near $\bar x$ and $v \in \partial f(x_k)$.
More formally, the function $f$ is \emph{semismooth} at $\bar x$~\cite{Mifflin77} whenever
\begin{align}\label{eq:semismoothclassical}
f(x) + \dotp{v, \bar x - x} = o(\|\bar x - x\|) \qquad \text{ as $x \rightarrow \bar x$ and $v \in \partial f(x)$},
\end{align}
where $o(\cdot)$ is any univariate function satisfying $\lim_{t \rightarrow 0} o(t)/t = 0$.
While it may at first seem stringent, semismoothness is a reasonable assumption since it holds for any locally Lipschitz weakly convex~\cite{Mifflin77} or semialgebraic function~\cite{BDL09}.

Turning to our main algorithm, we depart from the quadratic programming problem of~\eqref{eq:polyakreform2} and instead construct both our iterates $x_k$ and the collection $\{(y_i, v_i)\}_i$ by solving a sequence of linear systems, a simpler operation in general.
At iteration $k$, we construct the collection as follows:
%Turning to the construction of the collection $\{(y_i, v_i)\}_i$, we use the following simple procedure:  
set initial point $y_{0} = x_k$, choose subgradient $v_{0} \in \partial f(y_{0})$, and for $j = 1, \ldots, d$, recursively set
\begin{align}\label{eq:polyakreform3}
y_{j} &:= \argmin_{x \in \RR^d} \|x - x_k\|^2 \qquad \text{subject to: $\begin{bmatrix} f(y_{i}) + \dotp{v_{i}, x - y_{i}}\end{bmatrix}_{i=0}^{j-1} = 0$}
\end{align}
and choose $v_{j} \in \partial f(y_{j})$ arbitrarily.
For this collection, we will show that the next iterate
$$
x_{k+1} \in \argmin_{y \in \{y_i\}_i} f(y) \qquad \text{ satisfies } \qquad f(x_{k+1}) = o(f(x_k)) \quad \text{ as $k \rightarrow \infty.$}
$$
The construction of $y_i$ may at first seem mysterious, but its success results from a simple ``lemma of alternatives" proved in this work. Namely,
suppose that the first $j-1$ elements $y_1, \ldots, y_{j-1}$ do not superlinearly improve on $x_k$. Then we prove that one of the following must hold: either $y_j$ superlinearly improves upon $x_k$ or the rank of $[v_i^\T]_{i = 0}^{j}$ is $j+1$. 
In this way we must obtain local superlinear improvement in at most $d$ steps.

Thus, for sharp semismooth functions, simply repeating~\eqref{eq:polyakreform3} will result in superlinear convergence in a small, dimension dependent neighborhood of $\bar x$.
While this method converges superlinearly, its theoretical region of admissible initializers is small. 
Our numerical experiments suggest this may be a limitation of the analysis, rather than of the algorithm. 
Nevertheless, it is desirable to have a linearly convergent fallback method that quickly reaches the region of superlinear convergence from a much larger set of initial conditions.
To that end, we extend the linear convergence of the Polyak subgradient method~\eqref{eq:polyakreform} to sharp and semismooth functions (see~\cref{thm:polyak}). 
The argument and result mirror the previous result for weakly convex functions~\cite{DDMP18}.

While the Polyak algorithm eventually reaches the region of superlinear convergence, its entrance may be hard to detect.
Thus, we provide a generic procedure for coupling the superlinear steps~\eqref{eq:polyakreform3} with the Polyak algorithm~\eqref{eq:polyakreform} (or another fallback algorithm), which rapidly converges to the region of superlinear convergence when initialized in a much larger region.  The coupled algorithm may be implemented with knowledge of a single parameter, namely, the optimal value $f(\bar x)$. An intriguing open problem, left to future work, is whether one can design a parameter free variant.

The results stated thus far assume that $\cX_\ast$ is isolated at $\bar x$. We prove that all of the  algorithms analyzed in this work converge superlinearly to nonisolated solutions for functions that are \emph{$(b)$-regular} along  $\cX_\ast$, a natural uniformization of the semismoothness property that was recently analyzed in~\cite{DDJ21}. 
We review and provide several examples of the $(b)$-regularity property and develop a calculus for creating further examples, going beyond the setting of~\cite{DDJ21}.
For example, we show that a composition $f = h\circ F$ is $(b)$-regular along $\cX_\ast$ whenever $(i)$ $F$ is a smooth mapping and $(ii)$ $h$ is a locally Lipschitz semialgebraic function with isolated minimum $\bar y \in \range(F)$.
We use these results to provide useful corollaries for root-finding and feasibility problems  and discuss relations to the literature on semismooth Newton methods~\cite{kummer1988newton,qi1993convergence,facchinei2007finite,izmailov2014newton,klatte2006nonsmooth,qi1993nonsmooth} and accelerations of projection methods~\cite{Pang15,pang2015set}.

Finally, we note that despite local superlinear convergence, the worst-case complexity of the proposed method depends on $d$, a property not in line with the ``dimension free" complexity theory of first-order methods.  
Nevertheless, we found that we may terminate~\eqref{eq:polyakreform3} early in several scenarios, yielding promising empirical performance. 
For example, in Figure~\ref{fig:conditioning-hadamard} we plot the performance of the proposed method, dubbed $\superpolyak$, against the method~\eqref{eq:polyakreform}, dubbed $\polyak$, on a simple low-rank matrix sensing problem. Here the problem of interest is simply 
\begin{align*}
  f(U, V) := \frac{1}{m} \norm{\cA(UV^{\T}) - \cA(\bar M)}_1 \qquad \text{for all $U, V \in \RR^{d \times r}$},
\end{align*}
where $\bar M \in \RR^{d \times d}$ is a fixed rank $r$ matrix and $\cA: \Rbb^{d \times d} \to \Rbb^m$ is a linear operator; see Section~\ref{sec:matrix-sensing} for a more detailed description.
From the the plots, we see the proposed method performs well in terms of time and oracle complexity and appears less sensitive to the condition number $\tilde \kappa$ of the matrix $\bar M$.
Beyond early termination, we also introduce and use several other implementation strategies, including one that reduces the naive arithmetic complexity cost of constructing the points $y_i$ from $O(d^4)$ (ignoring subgradient evaluations) to $O(d^3)$ arithmetic operations. 
With these strategies in place, the advantage of $\superpolyak$ persists in several scenarios outlined in our numerical illustration. 
\begin{figure}[H]
  \centering
  \includegraphics[width=0.95\textwidth]{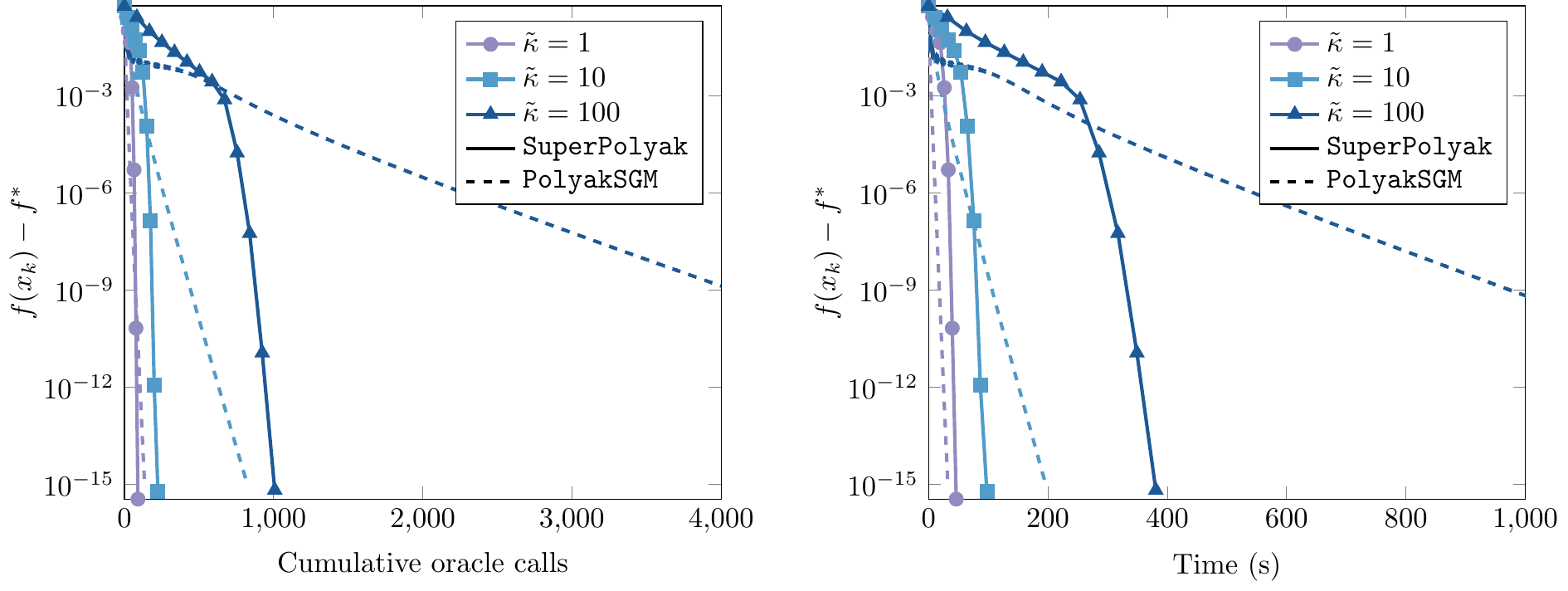}
  \caption{Low-rank matrix sensing with Hadamard measurements, varying condition number $\tilde \kappa$, and parameters $d = 2^{15}$, $r = 2$ and $m = 16d$.
  See~\cref{sec:matrix-sensing} for description.}
  \label{fig:conditioning-hadamard}
\end{figure}

Before turning to the formal statements of the results, the following section formalizes the basic notations and constructions used throughout this work.

\subsection{Notation and basic constructions}

We will mostly follow standard notation used in convex analysis as set out in the monograph~\cite{RW98}.
Throughout, the symbol $\R^d$ will denote a $d$-dimensional Euclidean space with the inner product
$\dotp{\cdot,\cdot}$ and the induced norm $\norm{x} = \sqrt{\dotp{x,x}}$. We denote the open ball of
radius $\varepsilon>0$ around a point $x\in \R^d$ by the symbol $B_{\varepsilon}(x)$. We use the symbol $\closedball$ to denote the closed unit ball at the origin.
A set-valued mapping $G \colon \RR^d \multo \RR^m$ maps points $x \in \RR^d$ to sets $G(x) \subseteq \RR^m$.
We say a set-valued mapping $G$ is nonempty-valued if $G(x)$ is nonempty for every $x \in \RR^d$ and locally bounded if $G(\cX) := \bigcup_{x \in \cX}G(x)$ is a bounded set for any bounded set $\cX \subseteq \RR^d$.
For any set $\cX\subseteq\R^d$, the  {\em distance function} and the {\em projection map} are defined by
\begin{equation*}
\dist(x,\cX):=\inf_{y\in \cX} \|y-x\|\qquad \textrm{and}\qquad
P_\cX(x):=\argmin_{y\in \cX} \|y-x\|,
\end{equation*}
respectively. 
Given a function $f \colon \RR^d \rightarrow \RR$ and $\gamma > 0$, we define the proximal operator $\prox_{\gamma f} \colon \RR^d \multo \RR^m$ of $f$ to be the set-valued mapping with values:
$$
\prox_{\gamma f}(x) := \argmin_{y} \left\{f(y) + \frac{1}{2\gamma}\|y - x\|^2\right\} \qquad \text{for all $x \in \RR^d$}.
$$
We call a function $h\colon \RR^d \rightarrow \RR$ \emph{sublinear} if its epigraph is a closed convex cone, and in that case we define
$$
\text{Lin}(h) = \{ x\in \RR^d\colon  h(x) = - h(-x)\}
$$
to be its \emph{lineality space.} Given a matrix $A$, we denote its spectral norm by $\opnorm{A}$. 

\paragraph{Semialgebraicity.} We call a set $\cX \subseteq \RR^d$ \emph{semialgebraic} if it is the union of finitely many sets defined by finitely many polynomial inequalities. Likewise, we call a function $f \colon \RR^d \rightarrow \RR$ semialgebraic if its graph $\gph(f) = \{(x, f(x)) \colon x\in \RR^d\}$ is semialgebraic. Finally, we call a set-valued mapping $G \colon \RR^d \multo \RR^m$ \emph{semialgebraic} if its graph $\gph(G) = \{(x,y) \colon y \in G(x)\}$ is semialgebraic.

\paragraph {Subdifferentials.}
Consider a locally Lipschitz function $f\colon\R^d\to\R$ and a point $x$. 
%The {\em Fr{\'e}chet subdifferential} of $f$ at $x$, denoted by $\partial_F f(x)$, consists of all vectors $v\in\R^d$ satisfying
%$$f(y)\geq f(x)+\langle v,y-x\rangle+o(\|y-x\|)\qquad \textrm{as }y\to x.$$
%The {\em Limiting subdifferential} of $f$ at $x$, denoted by $\partial_L f(x)$, consists of all vectors $v\in\R^d$ such that there exists sequence $x_i \in \RR^d$ and $v_i \in \partial_F f(x_i)$ satisfying $$(x_i, v_i) \rightarrow (x, v)  \text{ as $i \rightarrow \infty$.}$$
%The {\em Clarke subdifferential} of $f$ at $x$, denoted by $\partial f(x)$, consists of all convex combinations of limiting subgradients
%$$
%\partial f(x) = \text{conv} \partial_L \,f(x).
%$$
The Clarke subdifferential is the convex hull of limits of gradients evaluated at nearby points
$$
\partial f(x) = \text{conv} \left\{ \lim_{i \rightarrow \infty} \nabla f(x_i) \colon x_i \stackrel{\Omega}{\rightarrow} x\right\},
$$
where $\Omega \subseteq \RR^d$ is the set of points at which $f$ is differentiable (recall Radamacher's theorem). 
If $f$ is $L$-Lipschitz on a neighborhood $U$, then for all $x \in U$ and $v \in \partial f(x)$, we have $\|v\| \leq L$. A point $\bar x$ satisfying $0 \in \partial f(x)$ is said to be critical for $f$.
A function $f$ is called {\em $\rho$-weakly convex} on an open convex set $U$ if the perturbed function $f+\frac{\rho}{2}\|\cdot\|^2$ is convex on $U$. The Clarke subgradients of such functions automatically satisfy the uniform approximation property:
$$f(y)\geq f(x)+\langle v,y-x\rangle-\frac{\rho}{2}\|y-x\|^2\qquad \textrm{for all }x,y\in U, v\in \partial f(x).$$
Finally consider a locally Lipschitz mapping $F \colon \RR^d \rightarrow \RR^m$. Then the \emph{Clarke Jacobian of $F$ at $x$} is the set
$$
\partial F(x) = \text{conv} \left\{ \lim_{i \rightarrow \infty} \nabla F(x_i) \colon x_i \stackrel{\Omega}{\rightarrow} x\right\},
$$
where $\Omega \subseteq \RR^d$ is the set of points at which $F$ is differentiable. 

\paragraph{Normal cones.}
Let $\cX$ be a closed set and let $\bar x \in \cX$. The \emph{Fr{\'e}chet normal cone} to $\cX$ at $\bar x$, denoted by $N^F_{\cX}(\bar x)$, consists of all vectors $v \in \RR^d$ satisfying
$$
\dotp{v, x - \bar x} \leq o(\| x - \bar x\|) \qquad \text{ as $x \stackrel{\cX}{\rightarrow} \bar x$.}
$$
The \emph{Limiting normal cone} to $\cX$ at $\bar x$, denoted by $N^L_{\cX}(\bar x)$, consists of all vectors $v \in \RR^d$ such that there exist sequences $x_i \in \cX$ and $v_i \in N_{\cX}^F(x_i)$ satisfying 
$$
(x_i, v_i) \rightarrow (x, v) \text{ as $i \rightarrow \infty$.}
$$
The {\em Clarke normal cone} of $\cX$ at $x$, denoted by $N_{\cX}(x)$, consists of all convex combinations of limiting normal vectors
$$
N_{\cX}(x) = \text{cl}\,\text{conv}\,N^L_{\cX}(\bar x).
$$
The normal cone is related to the distance function as follows:
\begin{align}\label{eq:distfunctionsub}
\partial \dist(x, \cX) = \begin{cases}
\text{conv} \,\frac{x - P_{\cX}(x)}{\dist(x, \cX)} &\text{if $x \notin \cX$;} \\
 N_{\cX}(x) \cap \closedball &\text{otherwise.}
\end{cases}
\end{align}
Finally, we recall that whenever $x \notin \cX$ and $\hat x \in P_{\cX}(x)$, we have $\frac{x - \hat x}{\|x- \hat x\|} \in N_{\cX}(\hat x)$.

\paragraph{Manifolds.}
We will need a few basic results about smooth manifolds, which can be found in
the references~\cite{boumal2020introduction,lee2013smooth}. A set $\cM \subseteq
\RR^d$ is called a $C^p$ smooth manifold (with $p \geq 1$) if there exists a
natural number $m$, an open neighborhood $U$ of $x$,  and a $C^p$ smooth mapping
$F \colon U \rightarrow \RR^m$ such that the Jacobian $\nabla F(x)$ is surjective
and $\cM \cap U = F^{-1}(0)$. The tangent and normal spaces to $\cM$ at
$x \in \cM$ are defined to be $T_\cM(x) = \ker(\nabla F(x))$ and $N_{\cM}(x) =
T_{\cM}(x)^\perp = \range(\nabla F(x)^\ast)$, respectively. 
If $\cM$ is a $C^2$-smooth manifold around a point $\bar x$, then there exists $C > 0$ such that $y - x \in T_{\cM}(x) + C\|y - x\|^2 \closedball$ for all $x, y \in \cM$ near $\bar x$.

\section{Assumptions, algorithms, and main results}\label{sec:aam}

In this section, we introduce our assumptions, algorithms, and main results.
To that end, throughout this work we consider the problem
\begin{equation}
  \minimize_{x \in \RR^d} \; f(x),
  \label{eq:opt-problem}
\end{equation}
where $f: \RR^d \to \RR$ is a locally Lipschitz function with optimal value $f^\ast$. We denote
$
\cX_\ast = \argmin_{x \in \RR^d} f(x)
$
and assume that $\cX_\ast \neq \emptyset$. We also fix a point $\bar x \in \cX_\ast$ and a radius $\delta > 0$, which factor into our initialization assumptions.

\subsection{Main assumptions: sharpness and $(b)$-regularity}
In this section, we formalize our assumptions on the growth and semismoothness of $f$.  We additionally provide three concrete example problem classes. 

\subsubsection{Sharp growth}
Our first technical assumption is that $f$ grows sharply away from $\cX_\ast$:
\begin{enumerate}[label=$\mathrm{(A\arabic*)}$]
  \item \label{item:assumption:main:sharpness}{\bf (Sharpness)} There exists $\mu > 0$ such that the estimate
    \begin{align*}
      f(x) - f^\ast \geq \mu \, \dist(x, \cX_\ast)     \qquad\text{holds for all $x\in B_{\delta}(\bar x)$}.
   \end{align*}
 \end{enumerate}
 Assumption~\ref{item:assumption:main:sharpness} is a classical regularity condition known to ensure (local) linear convergence of subgradient methods in the (weakly) convex setting~\cite{DDMP18}.
Sharp growth is known to hold in a range of problems, most classically in feasibility formulations of linear programs~\cite{Hof52} (see also the survey~\cite{10.5555/3226650.3226819}).
Several contemporary problems also exhibit sharp growth, for example,
nonconvex formulations of low-rank matrix sensing and completion problems~\cite{CCD+21}.

\subsubsection{Semismoothness}
Our second technical assumption is that $f$ satisfies a ``uniform semismoothness" condition with respect to $\cX_\ast$:
\begin{enumerate}[label=$\mathrm{(A\arabic*)}$]
 \setcounter{enumi}{1}
  \item  \label{item:assumption:main:b} {\bf ($(b)$-regularity)} There exists a locally bounded nonempty-valued set-valued mapping
    $\subg \colon \RR^d \multo \RR^d$ and constants $\constb, \eta > 0$ such that the estimate
 \begin{align}
   |f(x) + \dotp{v, y - x} - f^\ast| \leq \constb \norm{y-x}^{1+\eta} \label{eq:bestimate}
  \end{align}
holds for all $x \in B_{\delta}(\bar x)$, $v \in \subg(x)$, and $y \in \cX_\ast \cap B_{2\delta}(\bar x)$.
\end{enumerate}
Assumption~\ref{item:assumption:main:b} is a uniformization of the classical semismoothness property~\eqref{eq:semismoothclassical} of~\cite{Mifflin77}. 
In particular, the estimate~\eqref{eq:bestimate} is identical to~\eqref{eq:semismoothclassical}
when $\cX = \{\bar x\}$, $\subg = \partial f$, and the term $\|\cdot\|^{1+\eta}$ is replaced with any univariate function $o(\cdot)$ satisfying $\lim_{t \rightarrow 0} o(t)/t = 0$.
We require the stronger error modulus $\|\cdot\|^{1+\eta}$ to deal with $\cX_\ast$ that are not singleton sets. 
Importantly
if $\cX_\ast$ is a singleton all results of this work easily generalize to  ``little-$o$" error. 
The recent work \cite{DDJ21} introduced the general $(b)$-regularity estimate, provided several basic examples, and developed a calculus, focusing on the mapping
$\subg = \partial f$.
In Section~\ref{sec:breg}, we recall and extend the results of \cite{DDJ21}, introducing new  examples and proving a formal chain rule for $(b)$-regularity.
The latter result ensures that $\subg$ computed by certain automatic differentiation schemes are valid generalized gradient mappings~\cite{bolte2021conservative}. 
\begin{figure}[!h]
	\centering
\centering\includegraphics[width=.5\textwidth]{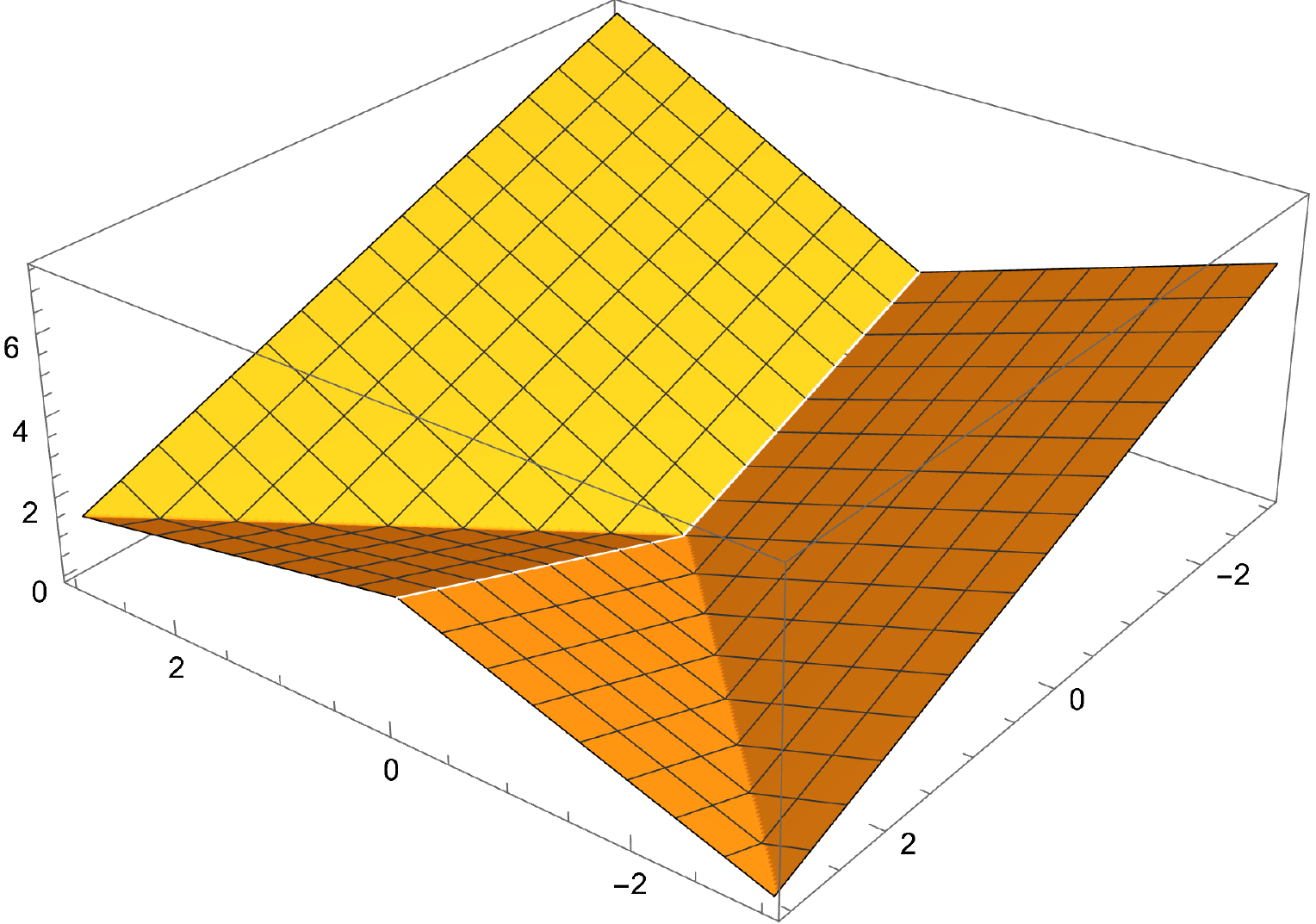}
\caption{The function $f(x,y) = |y- |x|| + \max(x, 0)$ satisfies Assumptions~\ref{item:assumption:main:sharpness} and~\ref{item:assumption:main:b} along the set $\cX_\ast = \{(x, -x)) \colon x \leq 0\}$.}\label{fig:sinfunction}
\end{figure}

\subsubsection{Examples}\label{sec:examplesofbregintro}
We now provide three concrete problem classes where $(b)$-regularity holds.
At the end of the section, we also touch upon sharp growth. 
We present the proofs of all three Propositions in Section~\ref{sec:bregfunctions}.

The first class arises from semialgebraic functions composed with smooth mappings. 
\begin{proposition}\label{prop:basicsemialgebraic}
Consider a $C^2$ smooth mapping $F \colon \RR^d \rightarrow \RR^m$ and a locally Lipschitz semialgebraic function $h \colon \RR^m \rightarrow \RR$. Suppose that $y\in \RR^m$ is an isolated minimum of $h$ and define $\cX_\ast = F^{-1}(y)$. Then for any $\bar x \in \cX_\ast$, the function 
$$
f(x) = h(F(x)) \qquad \text{for all $x \in \RR^d$},
$$ satisfies Assumption~\ref{item:assumption:main:b} along $\cX_\ast$ at $\bar x$ with mapping $\subg$ defined by the formal chain rule:
$$
  \subg(x) = \nabla F(x)^\T \partial h(F(x)) \qquad \text{for all $x \in \RR^d$}.
$$
\end{proposition}

A second class of examples arises from root finding problems.
\begin{proposition}\label{prop:rootfinding}
Consider a $C^2$ smooth mapping $F_1 \colon \RR^{d_1} \rightarrow \RR^{d_2}$ and a locally Lipschitz semialgebraic mapping $F_2 \colon \RR^{d_2} \rightarrow \RR^{d_3}$. Define $F := F_2 \circ F_1$ and the set $\cX_\ast := F^{-1}(0)$. Then for any $\bar x \in \cX_\ast$ at which $F_2(\bar x)$ is an isolated zero of $F_1$, the function  
$$
f(x) = \|F(x)\| \qquad \text{for all $x \in \RR^d$},
$$
satisfies Assumption~\ref{item:assumption:main:b} along $\cX_\ast$ at $\bar x$ with mapping $\subg$ defined by the formal chain rule:
$$
\subg(x) = \begin{cases}
\nabla F_1(x)^\T \partial F_2(x)^\T\frac{F(x)}{\|F(x)\|} &\text{$F(x) \neq 0$;}  \\
\nabla F_1(x)^\T \partial F_2(x)^\T  \closedball & \text{otherwise;}
\end{cases}\qquad \text{for all $x \in \RR^d$},
$$
where $\partial F_2(x)^\T$ denotes the set of transposed elements of the Clarke Jacobian of $F_2$ at $x$.
\end{proposition}
Finally we present a class arising in feasibility problems. 
\begin{proposition}\label{prop:semialgebraicfeasibility}
Consider a collection of semialgebraic sets $\cX_i \subseteq \RR^d$ indexed by a finite set $I$.
Suppose that $\bigcap_{i\in I}\cX_i = \{\bar x\}$ and define $\cX_\ast := \{\bar x\}$.
Then for any $\bar x \in \cX_\ast$, the function 
$$
f(x) = \sum_{i\in I}\dist(x, \cX_i)\qquad \text{for all $x \in \RR^d$},
$$
satisfies Assumption~\ref{item:assumption:main:b} along $\cX_\ast$ at $\bar x$ with mapping $\subg$ defined by the formal sum rule:
$$
\subg(x) = \sum_{i \in I} \partial \dist(x, \cX_i) \qquad \text{for all $x \in \RR^d$.}
$$
\end{proposition}

To close this section, we mention that the sharp growth property~\ref{item:assumption:main:sharpness} is well-studied in the settings of these propositions. 
For example, the setting of Proposition~\ref{prop:basicsemialgebraic} arises in  low-rank matrix estimation problems~\cite{CCD+21}, where regularity property~\ref{item:assumption:main:sharpness} is a consequence of the \emph{restricted isometry property}~\cite{candes2005decoding} of the ``measurement operator."
Next, for $f$ defined in Proposition~\ref{prop:rootfinding}, regularity property~\ref{item:assumption:main:sharpness} is simply the classical \emph{metric subregularity assumption}. This is a weak regularity property known to hold in many circumstances~\cite{ioffe2017variational,10.5555/3226650.3226819}.
Finally, for $f$ defined in Proposition~\ref{prop:semialgebraicfeasibility}, regularity property~\ref{item:assumption:main:sharpness} is simply the classical \emph{linear regularity assumption}, which is known to ensure local linear convergence of the alternating projection method for closed sets~\cite[Theorem 3.2.3]{drusvyatskiy2013slope}. The property is automatic, for example, for intersections of convex polyhedral sets (see~\cite[Fact 5.8]{bauschke2015linear}), and moreover holds for ``generic perturbations" of semialgebraic sets~\cite[Theorem 7.1]{drusvyatskiy2015transversality}. 
We present further analysis of these settings in Section~\ref{sec:discussion}.

\subsection{$\mathtt{PolyakSGM}$: local linear convergence}\label{sec:polyak}
We now turn to the first method of this work, dubbed $\mathtt{PolyakSGM}$, which is shown in Algorithm~\ref{alg:polyak-sgm-method}.
This method will be a key subroutine in the locally superlinearly convergent algorithm developed in~\cref{sec:coupling}. 
\begin{algorithm}[H]
  \caption{$\mathtt{PolyakSGM}(z_0, \epsilon)$}
  \begin{algorithmic}
    \Repeat \, for $\inner = 0, 1, \dots$
      \State Choose $v_{\inner} \in \subg(z_{\inner})$
      \If{$v_{\inner} = 0$}
        \State \Return $z_{\inner}$
      \EndIf
      \State $z_{\inner+1} := z_{\inner} - \dfrac{f(z_{\inner}) - f^\ast}{\norm{v_{\inner}}^2} v_{\inner}$
    \Until{$f(z_{{\inner}+1}) - f^\ast \leq \epsilon$}
    \State \Return $z_{{\inner}+1}$
  \end{algorithmic}
  \label{alg:polyak-sgm-method}
\end{algorithm}
The following is our main convergence theorem. We place the proof in Section~\ref{sec:thm:polyak}.
We note that the argument mirrors the proof of the analogous result in the convex and weakly convex
settings~\cite{Polyak69,DDMP18}.  
\begin{theorem}\label{thm:polyak}
Suppose assumptions \ref{item:assumption:main:sharpness} and ~\ref{item:assumption:main:b} hold at $\bar x$.
Let $L$ be an upper bound for the maximal norm element of $\subg(B_{\delta}(\bar x))$.
Define
$$
\kappa := \frac{L}{\mu} \quad
\text{and} \quad \rho :=  \sqrt{1 - (2\kappa)^{-2}}.
$$
Fix an initial point $x \in \RR^d$ satisfying the bounds:
\[
  \norm{x - \bar{x}} < \frac{(1-\rho)\delta}{2}
  \quad \text{and} \quad
  \dist(x, \cX_{\ast}) \leq \left(\frac{\mu}{4\constb}\right)^{1 / \eta}.
\]
Then for all $\epsilon > 0$, $\mathtt{PolyakSGM}(x, \epsilon)$ successfully terminates with at most
$$
\ceil{8\kappa^2 \log\left(\frac{\kappa (f(x) - f^{\ast})}{\epsilon}\right)}
$$
evaluations of $\subg$.
\end{theorem}

We note that it is possible to prove a similar theorem when the $(b)$-regularity estimate~\eqref{eq:bestimate} is replaced by the following weaker condition: for some $\gamma < \mu/2$, we have 
\begin{align}
f(x) + \dotp{v, y - x} - f^\ast \leq \gamma \|y - x\| 
\end{align}
for all $x \in B_{\delta}(\bar x)$ and $y \in \cX_\ast \cap B_{2\delta}(\bar x)$. 
We do not pursue this result since the stronger $(b)$-regularity estimate~\eqref{eq:bestimate} will be crucial in what follows.

\subsection{$\bundle$: superlinear improvement}\label{sec:subsuperlinear}

In this section, we formally describe the procedure outlined in Equation~\eqref{eq:polyakreform3} of the introduction. Specifically, we will show that the $\bundle$ procedure shown in Algorithm~\ref{alg:build-bundle-method} locally results in superlinear improvement. Note that pseudoinverse computations of Algorithm~\ref{alg:build-bundle-method} are identical to the subproblems in~\eqref{eq:polyakreform3}, but for ease of implementation, we have written the closed-form solution. 

\begin{algorithm}[H]
  \caption{$\bundle(x, \tau)$}
  \begin{algorithmic}
    \State $y_0 := x$; $v_0 \in \subg(y_0)$; $A_1 := v_0^\T$.
    \For{$\inner = 1, \dots, d$}
      \State $y_{\inner} = y_{0} - A_{\inner}^{\dag}
        \bmx{f(y_0) - f^\ast + \dotp{v_0, y_0 - y_0} \\
             \vdots \\
             f(y_{\inner-1}) - f^\ast + \dotp{v_{\inner-1}, y_0 - y_{\inner-1}}}$
      \State $A_{\inner+1} := \bmx{A_i \\ v_{\inner}^{\T}}$ for arbitrary $v_{\inner} \in \subg(y_{\inner})$.
    \EndFor
    \State \Return $y_{s}$, where $s = \argmin_{\inner: \norm{y_{\inner} - y_0} \le \tau f(y_0)} f(y_{\inner})$
  \end{algorithmic}
  \label{alg:build-bundle-method}
\end{algorithm}
Now we turn to our main theorem, which states that the procedure $\bundle$ locally results in superlinear improvement. The proof appears in Section~\ref{proof:corollary:function-value-reduction}.
\begin{theorem}[Superlinear Improvement]
\label[theorem]{corollary:function-value-reduction}
  Suppose Assumptions~\ref{item:assumption:main:sharpness} and~\ref{item:assumption:main:b} hold at $\bar x$. 
  Let $L$ be an upper bound for the maximal norm element of $\subg(B_{\delta}(\bar x))$
  and a 
  Lipschitz constant of $f$ on $B_{\delta}(\bar x)$.
  Then there exists a constant $\constsuper > 0$ such that for all scalars $\tau > (3/\mu)$ and points $x \in \RR^d$ with
  \[
    \norm{x - \bar{x}} < \frac{\delta}{4}, \quad \text{ and } \quad
    \dist(x, \cX_{\ast}) \leq \min\set{\left(\frac{\mu}{2\constb}\right)^{1/\eta},
      \left(\frac{\mu^{1-\eta}}{L \constsuper}\right)^{1 / \eta}
    },
  \]
  the point $\tilde x = \bundle(x, \tau)$ satisfies
  \begin{equation}
    \label{eq:bigcprime}
    f(\tilde{x}) \leq \constsuper f(x)^{1+\eta}.
  \end{equation}
\end{theorem}

We comment on two aspects of this theorem. First, we mention that the requirement that $\|\tilde x - y_0\| \leq \tau f(y_0)$ is not necessary for one step of superlinear improvement. However, in Section~\ref{sec:coupling} we apply $\bundle$ repeatedly and use this condition to ensure $\tilde x$ remains near $\bar x$. Second, upon checking the proof, the reader will find that the constant $\constsuper$ can be extremely large, yielding a small region of superlinear convergence: 
\[
  \constsuper = O\left(d\max\set{(L/\mu), L} (8\sqrt{2} L/\mu)^{d}\right).
\]
However, the numerical experiments in Section~\ref{sec:experiments} suggest that this bound may be an artifact of the proof technique.
Whether this constant can be improved is an intriguing open question.
In Section~\ref{sec:coupling} we develop a procedure that reaches the region of superlinear convergence from a more reasonable initial guess, using a reasonable number of evaluations of $\subg$. After it reaches this region, the method reverts to $\bundle$.

\begin{remark}
We mention that a naive implementation of $\bundle$ requires $O(d^4)$ operations.
In Section~\ref{sec:implementation}, we develop a strategy that reduces this cost to $O(d^3)$.
\end{remark}

\subsection{$\mathtt{SuperPolyak}$: $\bundle$ with a fallback algorithm}
\label[section]{sec:coupling}

In Section~\ref{sec:subsuperlinear}, we showed that the $\bundle(x, \tau)$ procedure results
in superlinear improvement if $\|x - \bar x\| \leq \delta/4$ and $\dist(x, \cX_\ast)$ is  small.
While the former condition is reasonable, the latter appears difficult to satisfy.
Thus, in this section, we develop a strategy for coupling the $\bundle$ procedure with a linearly convergent fallback algorithm, which rapidly approaches $\cX_\ast$ from a more reasonable initialization.

\subsubsection{Fallback algorithms}
The method $\mathtt{PolyakSGM}$ may always be used as a fallback method. 
However, an alternative fallback method may be preferable to $\mathtt{PolyakSGM}$. Two settings of interest arise from fixed-point and feasibility problems. 
\begin{example}[Fixed-point problems]\label{example:fixed-point}
\rm Suppose we seek a fixed-point $\bar x$ of a locally Lipschitz mapping $T \colon \RR^d \rightarrow \RR^d$. Then, under the conditions outlined in Proposition~\ref{prop:rootfinding}, one may apply $\mathtt{PolyakSGM}$ to the function $f(x) = \|x - Tx\|$. In place of~$\mathtt{PolyakSGM}$, one may instead use the classical fixed-point iteration~\cite{Banach22,Krasnoselskii55,Mann1953}, which repeats
$$z_{\inner+1} = T(z_{\inner}).$$ 
\end{example}

\begin{example}[Feasibility problems]\label{example:feasibility}
\rm Suppose we seek a point $\bar x$ in the intersection of two closed subsets $\cX_1$ and $ \cX_2$ of $\RR^d$. Then under the conditions outlined in Proposition~\ref{prop:semialgebraicfeasibility}, one may apply $\mathtt{PolyakSGM}$ to the function $f(x) = \dist(x, \cX_1) + \dist(x, \cX_2)$. In place of~$\mathtt{PolyakSGM}$, one may instead use the classical method of alternating projections~\cite{von}, which repeats 
$$
\tilde z_{\inner} \in P_{\cX_1}(z_i); \qquad z_{\inner+1} \in P_{\cX_2}(\tilde z_{\inner});
%z_{\inner + 1} \in P_{\cX_2}(P_{\cX_1}(z_{\inner})).
$$ 
\end{example}

 Although other fallback methods may be appropriate, we limit our study to algorithms which iterate algorithmic mappings of the following form: 
\begin{enumerate}[label=$\mathrm{(A\arabic*)}$]
  \setcounter{enumi}{2}
  \item \label{item:assumption:main:fallback} {\bf (Algorithmic mapping)}
    There exists radii $0 < \radius_2 \leq \radius_1$,  a contraction factor
    $\rho \in (0, 1)$, and a mapping
    $\cA \colon B_{\radius_1}(\bar x) \rightarrow \RR^d$  such that if $x \notin \cX_\ast$ satisfies
    \[
      \norm{x - \bar{x}} < \radius_1 \qquad \text{ and } \qquad 
      \dist(x, \cX_{\ast}) < \radius_2,
    \]
    then the following holds:
    \[
      \norm{\cA(x) - \hat x} \leq \rho \, \dist(x, \cX_\ast) \quad
      \text{ for all $\hat{x} \in P_{\cX_{\ast}}(x)$.}
    \]
\end{enumerate}

We call such mappings $\cA$ \emph{algorithmic mappings.} For example, we will later show in Lemma~\ref{lemma:one-step-improvement} that $\mathtt{PolyakSGM}$ is generated by iterating an algorithmic mapping.
In the context of Example~\ref{example:fixed-point}, the operator $\cA := T$ is an algorithmic mapping if it behaves like a contraction towards points in $\cX_\ast$. 
Finally, in the context of Example~\ref{example:feasibility}, \cite[Theorem 3.2.3]{drusvyatskiy2013slope} shows that any selection $\cA$ of the set-valued mapping $P_{\cX_2}\circ P_{\cX_1}$ is an algorithmic mapping provided the sets $\cX_1$ and $\cX_2$ intersect ``transversely" at $\bar x$, a property that implies sharp growth of $f$ (see~\cite{drusvyatskiy2015transversality} for discussion).
%Thus, algorithmic mappings generate well-defined iterates that eventually reach $\cX_\ast$.

Now consider the iterates $z_0 := \cA^{\circ k}(z_0)$ generated by repeatedly applying an algorithmic mapping, starting from some initial point $z_0$ that is sufficiently close to $\bar x$.
Then it is straightforward to check that the iterates $z^k$ linearly converge $\cX_\ast$; we provide a simple proof in Lemma~\ref{lemma:fallback-stay-in-ball} of Section~\ref{sec:proofofalgsresults}.
Thus, any such algorithmic mapping $\cA$ generates a well-defined algorithm $\mathtt{FallbackAlg}(\cA, z_0, \epsilon)$, which arises from simply iterating $\cA$ until the function gap is of size at most $\epsilon$ (see Algorithm~\ref{alg:fallback}). Such fallback methods play a key role in our main algorithm, which we now describe.

\begin{algorithm}
  \caption{$\mathtt{FallbackAlg}(\cA, z_0, \epsilon)$}
  \begin{algorithmic}
    \Repeat \, for $\inner = 0, 1, \dots$
      \State $z_{\inner+1} := \cA(z_{\inner})$
    \Until{$f(z_{\inner+1}) - f^{\ast} \leq \epsilon$}
    \State \Return $z_{\inner+1}$
  \end{algorithmic}
  \label{alg:fallback}
\end{algorithm}

\subsubsection{Algorithm and main convergence theorem}
We now have all the pieces to describe Algorithm~\ref{alg:bundle-newton-method}, which we dub $\mathtt{SuperPolyak}$. The method couples $ \bundle$ and $\mathtt{FallbackAlg}$. At each iteration it first attempts a superlinear step. If the step halves the function gap, the method updates the iterate. Otherwise, the method calls the fallback algorithm, which will halve the function gap.  
Key to the algorithm is the scalar $(3/2)^k$ in line~\ref{alg:threehalf}, which is eventually larger than $3 / \mu$: 
according to~\cref{corollary:function-value-reduction}, this ensures the $\bundle(x_{\outind}, (3/2)^{\outind})$ locally results in superlinear improvement.
Finally, we mention that one may adjust the performance of the algorithm by changing the factor $(3/2)^k$ or adjusting the constant $1/2$ in line~\ref{alg:fallbackstep}. 
We discuss these strategies in Section~\ref{sec:implementation} below. 

\begin{algorithm}
  \caption{$\mathtt{SuperPolyak}(\cA, x_0, \epsilon)$}
  \begin{algorithmic}[1]
    \For{$\outind = 0, 1, \dots$}
      \State $\tilde{x} := \bundle(x_{\outind}, (3/2)^{\outind})$ 
      \Comment{$\tilde{x}$ can be $\emptyset$} \label{alg:threehalf}
      \If{$\tilde{x} \neq \emptyset$ and $f(\tilde{x}) - f^\ast < \frac{1}{2} (f(x_{\outind}) - f^\ast)$}
        \label{op:progress}
        \State $x_{\outind+1} := \tilde{x}$ \Comment{$\bundle$ step successful}
      \Else
      \State $x_{\outind+1} := \mathtt{FallbackAlg}\left(\cA, x_{\outind}, \frac{1}{2} (f(x_{\outind}) - f^\ast)\right)$
        \Comment{Run until function gap halved}\label{alg:fallbackstep}
      \EndIf
      \If{$f(x_{\outind+1}) - f^\ast \le \epsilon$}
        \State \Return $x_{\outind+1}$
      \EndIf
    \EndFor
  \end{algorithmic}
  \label{alg:bundle-newton-method}
\end{algorithm}

The following theorem shows Algorithm~\ref{alg:bundle-newton-method} eventually results in superlinear improvement.
We place the proof in Section~\ref{sec:SuperPolyak}. 
\begin{thm}\label[theorem]{thm:SuperPolyak}
Suppose Assumptions~\ref{item:assumption:main:sharpness},~\ref{item:assumption:main:b},
and~\ref{item:assumption:main:fallback} hold at $\bar x$ for some algorithmic mapping $\cA$ with contraction factor $\rho$ and radii $\radius_1$ and $\radius_2$. Let $L$ be an upper bound for the maximal norm element of
$\subg(B_{\delta}(\bar x))$ and a Lipschitz constant of $f$ on $B_{\delta}(\bar x)$.
Define
$$
\kappa := \frac{L}{\mu}; \qquad  \Delta := f(x_0) - f^\ast.
$$
Fix an initial point $x\in \RR^d$ satisfying the bounds:
\begin{align}
  \begin{aligned}
    \norm{x_0 - \bar{x}} &\leq
    \left\{ \left( \frac{2}{1 - \rho} \right)\left(
      1 + \max\set{2 \kappa \, \frac{1 + \rho}{1 - \rho}, \frac{4L}{3}}\right)\right\}^{-1}
      \min\set{{\frac{\delta}{4}}, \radius_1};\\
    \dist(x_0, \cX_{\ast}) &\leq
    \frac{\radius_2}{1 + \max\set{2\kappa \, \frac{1 + \rho}{1 - \rho}, \frac{4L}{3}}}.
  \end{aligned}\label{eq:initialcondfinal}
\end{align}
Define the constant (where $\constsuper$ appears in Theorem~\ref{corollary:function-value-reduction})
\begin{align}
  K_1 := \ceil{\max\left\{\log_2\left(
  \Delta \cdot \max\set{
      2(2\constsuper)^{1/\eta}, \frac{(2\constb)^{1/\eta}}{\mu^{1 + 1/\eta}},
      (\kappa\constsuper)^{1/\eta}}
  \right), \log_{\frac{3}{2}} \left( \frac{3}{\mu} \right) \right\}}.\label{eq:k1}
\end{align}
Then for any $\epsilon > 0$, Algorithm~\ref{alg:bundle-newton-method} successfully terminates with at most
\begin{enumerate}
 \item $\ceil{\frac{1}{1-\rho} \log(2\kappa)} K_1 $ evaluations of $\cA$;
\item $d K_1
  + d \ceil{\frac{\log \log_2 \left(1/\epsilon\right)}{\log(1 + \eta)}}$ evaluations of $g$.
\end{enumerate}
\end{thm}

The following corollary examines the complexity of Algorithm~\ref{alg:bundle-newton-method} when the fallback method arises from the Polyak subgradient method. 
In this setting, evaluating $\cA$ requires evaluating both $G$ and $f$ once.
We place the proof the following corollary in Section~\ref{sec:corsuperpolyak}.
\begin{corollary}\label{cor:corsuperpolyak}
  Consider the setting of Theorem~\ref{thm:SuperPolyak}. Suppose that the fallback algorithm is $\mathtt{PolyakSGM}(x, \epsilon)$. Define $\rho := \sqrt{1-(2\kappa)^{-2}}$ and suppose that
  \begin{align*}
  \begin{aligned}
    \norm{x_0 - \bar{x}} &\leq
    \left\{ \left( \frac{2}{1 - \rho} \right)\left(
      1 + \max\set{2 \kappa \, \frac{1 + \rho}{1 - \rho}, \frac{4L}{3}}\right)\right\}^{-1}
     {\frac{\delta}{4}};\\
    \dist(x_0, \cX_{\ast}) &\leq
  \frac{1}{1 + \max\set{2\kappa \, \frac{1 + \rho}{1 - \rho}, \frac{4L}{3}}}
  \left(\frac{\mu}{4\constb}\right)^{1/\eta}.
  \end{aligned}
  \end{align*}
  Then~\cref{alg:bundle-newton-method} will successfully terminate after at most
  \[
    \max\set{d, \ceil{8\kappa^2 \log(2\kappa)}} K_1
    + d \ceil{\frac{\log \log_2 \left(\frac{1}{\epsilon}\right)}{
    \log\left(1 + \eta\right)}}
  \]
  evaluations of $\subg$, where $K_1$ appears in~\eqref{eq:k1}.
\end{corollary}

We now turn our attention to further consequences of Theorem~\ref{thm:SuperPolyak}.

\subsection{Consequences for root-finding and feasibility problems}\label{sec:discussion}

In this section, we describe consequences of Theorem~\ref{thm:SuperPolyak} for root-finding and feasibility problems -- two settings where the optimal value $f^\ast$ is known and equal to zero. 
For both problem classes, we consider a simple scenario and discuss related literature.
Further extensions are possible. For example, we may consider more complex problem structure using the calculus results of the upcoming Section~\ref{sec:breg}. We may also use further generalized gradient maps $\subg$. 
We omit these extensions for brevity.

\subsubsection{Root-finding problems}

We have the following corollary for root-finding problems. We place the proof in Section~\ref{sec:cor:newtonsemismooth}.
\begin{corollary}\label{cor:newtonsemismooth}
Let $F \colon \RR^d \rightarrow \RR^m$ be a locally Lipschitz mapping. Define $\cX_\ast = F^{-1}(0)$ and let $\bar x \in \cX_\ast$. Fix $\mu, \eta > 0$ and assume that 
\begin{enumerate}
\item \label{cor:newtonsemismooth:metric}$F$ is $\mu$-metrically subregular at $\bar x$, meaning 
$$
\|F(x)\| \geq \mu\, \dist(x, \cX_\ast) \qquad \text{ for all $x$ near $\bar x$.}
$$
\item \label{cor:newtonsemismooth:b} $(F, \partial F)$ is $(b)$-regular along $\cX_\ast$ at $\bar x$ with exponent $1+\eta$, meaning there exists $C > 0$ such that the estimate
\begin{align*}%\label{eq:bregestimate}
|F(x) +A(y - x)| \leq C\|y- x\|^{1+\eta}
\end{align*} 
holds for all $x$ near $\bar x$, $A \in \partial F(x)$, and $y \in \cM$ near $\bar x$.
\end{enumerate}
In particular, Item~\ref{cor:newtonsemismooth:b} is automatically satisfied when $\cX_\ast$ is isolated at $\bar x$ and $F$ is semialgebraic. Now define a function and generalized gradient mapping: for all $x \in \RR^d$,
$$
f(x) := \|F(x)\| \qquad \text{ and } \qquad \subg(x) := \begin{cases}
\partial F(x)^\T \frac{F(x)}{\|F(x)\|};\\
\partial F(x)^\T  \closedball.
\end{cases}
$$
Then $f$ and $\subg$ satisfy assumptions \ref{item:assumption:main:sharpness} and~\ref{item:assumption:main:b}. 
Therefore, Algorithm~\ref{alg:bundle-newton-method} with fallback method $\mathtt{PolyakSGM}$ locally superlinearly converges to a root of $F$. 
\end{corollary}

We now place this result in the context of the so-called ``semismooth Newton" method, which has a vast literature, summarized in the seminal papers and monographs~\cite{kummer1988newton,qi1993convergence,facchinei2007finite,izmailov2014newton,klatte2006nonsmooth,qi1993nonsmooth}. 
To focus our discussion, we compare and contrast Corollary~\ref{cor:newtonsemismooth} with the results of~\cite{qi1993nonsmooth}.
The semismooth Newton method of~\cite{qi1993nonsmooth} directly generalizes the classical Newton method to nonsmooth equations, replacing the classical Jacobian with an element of the Clarke Jacobian.
For simplicity we describe this method for square systems $F = 0$, where $F \colon \RR^d \rightarrow \RR^d$ is a locally Lipschitz mapping. 
To solve this equation, the pioneering work of Qi and Sun~\cite{qi1993nonsmooth} considers the following assumptions near a root $\bar x$:
\begin{enumerate}
%\item {\bf (Isolated solution)} $\bar x$ is an isolated root of $F$
\item {\bf (Invertibility)} every $A \in \partial F(\bar x)$ is invertible (in particular $\bar x$ is isolated).
\item {\bf (Semismoothness)} $F$ is semismooth at $\bar x$, meaning 
\begin{align*}%\label{eq:bregestimate}
|F(x) +A(\bar x - x)| = o(\|\bar x -  x\|) \qquad \text{ for all $A \in \partial F(x)$ as $x \rightarrow \bar x$.}
\end{align*} 
\end{enumerate}
Under these assumptions, the work~\cite{qi1993nonsmooth} shows that the semismooth Newton iteration 
\begin{align}\label{eq:semismooth}
x_{k+1} = x_k - A_k^{-1}  F(x_k) \qquad \text{ for some $A_k \in \partial F(x_k)$.}
\end{align}
is locally well-defined and the iterates $x_k$ converge superlinearly to $\bar x$. 
Much work on semismooth Newton methods considers similar conditions to the work of  Qi and Sun~\cite{qi1993nonsmooth}. 
While semismoothness is in some sense minimal, 
it is desirable to weaken the invertibility condition to the metric subregularity condition of Corollary~\ref{cor:newtonsemismooth}.
Such a result would be useful for the acceleration of certain first-order methods for signal recovery, which may be represented by the fixed-point iteration of Example~\ref{example:fixed-point}. In particular, it is known that the \emph{proximal gradient operator} associated to certain {compressive sensing} problems~\cite{cand_tao,Don06} is metrically subregular, but does not satisfy the stronger invertibility condition (see Section~\ref{sec:compressedsensing} for a description of the problem).
To the best of our knowledge, Corollary~\ref{cor:newtonsemismooth} presents the first semismooth Newton-type method that converges under the metric subregularity condition, even in the case of an isolated solution of a general semismooth mapping $F$.

Finally, we mention two further semismooth Newton-type methods that succeed under the metric subregularity condition, but require further assumptions. First, the SuperMann scheme of~\cite{8675506} proposes a nonsmooth (quasi) Newton scheme that converges superlinearly under semi-differentiability and metric subregularity if certain inverse Hessian approximations remain bounded throughout the developed algorithm; the latter property is nontrivial and not verified in that work. Second, the LP-Newton method of~\cite{facchinei2014lp} proposes a Newton-type methods that converges superlinearly under metric subregularity if a certain smoothness assumption holds; the assumption appears stronger than the classical semismoothness assumption considered in this work.

\subsubsection{Feasibility problems}

We have the following corollary for feasibility problems. We place the proof in Section~\ref{sec:cor:feasibilityfinal}.

\begin{corollary}\label{cor:feasibilityfinal}
Consider a collection of closed sets $\cX_i \subseteq \RR^d$ indexed by a finite set $I$.
Define $\cX_\ast := \bigcap_{i\in I}\cX_i$ and let $\bar x \in \cX_\ast$.
Fix $\mu, \eta, C > 0$ and suppose that 
\begin{enumerate}
\item  \label{cor:feasibilityfinal:regular} The family $\{\cX_i\}_i$ is $\mu$-linearly regular at $\bar x$, meaning
\begin{align*}
\sum_{i \in I} \dist(x, \cX_i)  \geq \mu \, \dist(x, \cX_\ast) \qquad \text{for all $x$ near $\bar x$.}
\end{align*}
\item \label{cor:feasibilityfinal:semismoothness} For all $i \in I$, we have 
$$
|\dotp{v, y - x}| \leq C \|v\|\|y-x\|^{1+\eta}, 
$$
for all $x \in \cX_i$ and $y \in \cX_\ast$ near $\bar x$ and all $v \in N_{\cX_i}(x)$.
\end{enumerate}
In particular, Item~\ref{cor:feasibilityfinal:semismoothness} is automatically satisfied when either (i) $\cX_i$ is a $C^2$ manifold for all $i \in I$ or (ii) $\cX_\ast$ is isolated at $\bar x$ and $\cX_i$ is semialgebraic or a $C^2$ smooth manifold around $\bar x$ for $i \in I$. 
Now define a function and generalized gradient mapping: for all $x \in \RR^d$,
$$
f(x) = \sum_{i \in I}\dist(x,\cX_i) \qquad \text{ and } \qquad \subg(x) =
\sum_{i\in I} \partial \dist(x, \cX_i),
$$
Then $f$ and $\subg$ satisfy assumptions \ref{item:assumption:main:sharpness}  and~\ref{item:assumption:main:b}.
Therefore, Algorithm~\ref{alg:bundle-newton-method} with fallback method $\mathtt{PolyakSGM}$ locally superlinearly converges to $\cX_\ast$.
\end{corollary}

Some comments are in order. We note that Item~\ref{cor:feasibilityfinal:semismoothness} is a natural notion of $(b)$-regularity for nested sets $\cX_\ast \subseteq \cX_i$. We comment more on its history and examples in Section~\ref{sec:examplesbreg}.
Next we discuss related work.
The most related results in the literature are developed in~\cite{Pang15,pang2015set}. 
The work~\cite{Pang15} in particular develops a superlinearly convergent procedure for nonconvex feasibility problems, which solves a quadratic programming problem at each iteration.
The algorithm is shown to converge superlinearly when the classical \emph{transversality} property holds 
\begin{align}\label{eq:regularitynormalcones}
\text{If } \sum_{i \in I} v_i = 0 \text{, for $v_i \in N_{\cX_i}(\bar x)$, then } v_i = 0 \text{ for $i \in I$,} 
\end{align}
and either of the following two conditions hold for $i \in I$:
\begin{enumerate}
\item the set $\cX_i$ is a manifold;
\item the normal cone to $\cX_i$ has a unique unit norm element near $\bar x$.
\end{enumerate}
The first setting is most interesting. In this case, the following corollary holds.
\begin{corollary}\label{cor:manifoldsuper}
Consider the setting of Corollary~\ref{cor:feasibilityfinal}. Suppose that the set $\cX_i$ is a $C^2$ manifold for all $i \in I$ and that the family intersects transversely at $\bar x$ in the sense of~\eqref{eq:regularitynormalcones}.
Then Items~\ref{cor:feasibilityfinal:regular} and~\ref{cor:feasibilityfinal:semismoothness} (with $\eta = 1$) of Corollary~\ref{cor:feasibilityfinal} hold.
\end{corollary}
\noindent Thus, in the case of transversal manifold intersections, Algorithm~\ref{alg:bundle-newton-method} converges superlinearly (in fact, quadratically) under the same setting as~\cite{Pang15}. 
Beyond the manifold setting, Corollary~\ref{cor:feasibilityfinal} provides  additional consequences for semialgebraic intersections.
Finally, we mention one benefit of \cref{alg:bundle-newton-method} compared to the algorithm of~\cite{Pang15}: each step solves a linear system, rather than a quadratic programming problem. 

\paragraph{Outline of the rest of the paper.} Having stated all of our main results, we now turn to proofs and a brief numerical study. First, Section~\ref{sec:breg} studies the $(b)$-regularity property, providing basic examples and proving calculus rules.  Next, Section~\ref{sec:proofofalgsresults} proves all the algorithmic results stated in this section. 
Finally, Section~\ref{sec:experiments} presents a brief numerical study and describes several implementation strategies.

\section{The $(b)$-regularity property: examples and calculus}\label{sec:breg}

In this section, we present basic examples and calculus for the $(b)$-regularity property in Assumption~\ref{item:assumption:main:b}. This property was recently studied in the manuscript~\cite{DDJ21}, focusing on functions and the Clarke subdifferential. In the following definition, we broaden the concept to mappings.\footnote{Note the slight discrepancy with Assumption~\ref{item:assumption:main:b}: in the terminology of this section, the pair $(f, \subg^\T)$ is $(b)$-regular along $\cX_\ast$ at $\bar x$ with exponent $1+\eta$.}

\begin{definition}[$(b)$-regularity along a set $\cY$]\label{def:breg}
\rm Consider a locally Lipschitz mapping $F\colon \RR^d \rightarrow \RR^m$, a set $\cY$
and a nonempty-valued $\jac \colon \RR^d \multo \RR^{m \times d}$. Fix a point $\bar x \in \cY$ and a scalar $\eta > 0$. Then the pair $(F, \jac)$ is \emph{$(b)$-regular along $\cY$ at $\bar x$ with exponent $1+\eta$} if there exists $C > 0$ such that the estimate
\begin{align}\label{eq:bregestimate}
\|F(x) +A(y - x) - F(y)\| \leq C\|x - y\|^{1+\eta}
\end{align} 
holds for all $x$ near $\bar x$, $A \in \jac(x)$, and $y \in \cY$ near $\bar x$.
\end{definition}

\subsection{Examples}\label{sec:examplesbreg}

A natural choice for the mapping $\jac$ in Definition~\ref{def:breg} is simply the Clarke Jacobian: $\jac = \partial F$. More generally, ``generalized Jacobian mappings" can arise from automatic differentiation routines. Recently, Bolte and Pauwels~\cite{bolte2021conservative} developed a mathematical model for such routines. In their work they identified that the output of such routines  are often \emph{conservative set-valued vector fields,} as formalized in the following definition.

\begin{definition}[Conservative set-valued vector fields.]
{\rm Consider a locally Lipschitz mapping $F\colon \RR^d \rightarrow \RR^m$ and a set-valued mapping $\jac \colon \RR^d \multo \RR^{m \times d}$ with nonempty compact-values and closed graph. Then $\jac$ is a \emph{conservative set-valued vector field for $F$} if for any absolutely continuous curve $x \colon [0,1] \rightarrow \RR^d$, we have 
\begin{align*}
\frac{d}{dt} F(x(t)) = A x(t) \qquad \text{for a.e. $t \in [0, 1]$ and all $A \in \jac(x(t))$}.
\end{align*}}
\end{definition}

As shown by~\cite{bolte2021conservative} (with precursors in~\cite{drusvyatskiy2015curves,davis2020stochastic}), the Clarke Jacobian $\partial F$ is a conservative set-valued vector field for any semialgebraic mapping $F$, though other examples are possible~\cite{lewis2021structure}. 
The later work~\cite{DD21} then showed that whenever both $F$ and $\jac$ are semialgebraic, conservative set-valued vector fields satisfy the $(b)$-regularity along singleton sets. This is quoted in the following lemma, consisting of several basic examples of Definition~\ref{def:breg}.
\begin{lemma}[Basic Examples]\label{lem:basicexamples}
Suppose that $F \colon \RR^d \rightarrow \RR^m$ is a locally Lipschitz mapping. Fix a point $\bar x \in \RR^d$ and closed sets $\cY\subseteq \RR^d$ containing $\bar x$.
\begin{enumerate}
\item {\bf (Smooth mappings)} \label{lem:basicexamples:item:c1} If $F$ is $C^1$ near $\bar x$ and the Jacobian $\nabla F$ is locally Lipschitz, then the pair $(F, \nabla F)$ is (b)-regular along $\cY$ at $\bar x$ with exponent $2$. 
\item {\bf (Sublinear functions)} \label{lem:basicexamples:item:sublinear} If $m = 1$, the mapping $F$ is sublinear, and $\cY = \mathrm{Lin}(F)$, then for all $\eta > 0$ the pair $(F, \partial F)$ is $(b)$-regular along $\cY$ at $\bar x$ with exponent $1+\eta$.
\item {\bf (Semialgebraic mappings)} \label{lem:basicexamples:item:semi} If $F$ is semialgebraic, $\cY = \{\bar x\}$, and $\jac$ is a semialgebraic conservative set-valued vector field for $F$ (e.g., $\jac = \partial F$), then there exists $\eta > 0$ such that the pair $(F, G)$ is $(b)$-regular along $\cY$ at $\bar x$ with exponent $1+\eta$.
\end{enumerate}
\end{lemma}
\begin{proof}
Item~\ref{lem:basicexamples:item:c1} is straightforward so we omit the proof. Item~\ref{lem:basicexamples:item:sublinear} is the shown in~\cite[Lemma 2.6.1]{DDJ21}. 
Item~\ref{lem:basicexamples:item:semi} is shown in~\cite{DD21} (the case of $\jac =  \partial F$ is shown in~\cite{BDL09}). 
\end{proof}

Finally we give several examples involving distance functions. 
Here we present a key sufficient condition -- Equation~\eqref{eq:semismoothsets}. 
This condition, which was first introduced in~\cite{whitney1992local,whitney1992tangents} for manifolds and recently studied for general sets in \cite{DDJ21}, is the classical notion of  $(b)$-regularity for two nested sets.
\begin{lemma}[Distance Functions]\label{lem:bregdistance}
Fix a point $\bar x \in \RR^d$ and closed sets $\cY \subseteq \cX$ in $\RR^d$ containing $\bar x$. Consider the following condition: there exists $C, \eta > 0$ such that  
\begin{align}\label{eq:semismoothsets}
|\dotp{v, y - x}| \leq C \|v\|\|y-x\|^{1+\eta},
\end{align}
for all $x \in \cX$ and $y \in \cY$ near $\bar x$ and all $v \in N_{\cX}(x)$.
Define $G := \partial \dist(\cdot, \cX)^\T$.
Then the pair $(\dist(\cdot, \cX), G)$ is $(b)$-regular along $\cY$ at $\bar x$ with exponent $1+\eta$ if and only if~\eqref{eq:semismoothsets} holds.
In particular,~\eqref{eq:semismoothsets} holds automatically when 
\begin{enumerate}
\item \label{lem:bregdistance:semialgebraic} $\cX$ is semialgebraic and $\cY = \{\bar x\}$.
\item \label{lem:bregdistance:manifold} $\cX = \cY$ and $\cX$ is a $C^2$ manifold around $\bar x$ (with $\eta = 1$).
\item\label{lem:bregdistance:cone} $\cX$ is a convex cone and $\cY = \cX \cap (-\cX)$ is its lineality space (with $\eta = 1$).
\end{enumerate}
\end{lemma}
\begin{proof}
Note that whenever the pair is $(b)$-regular, the estimate~\eqref{eq:semismoothsets} trivially follows from~\eqref{eq:distfunctionsub}.
Now we prove that~\eqref{eq:semismoothsets} implies the pair is $(b)$-regular. To that end, let $\delta > 0$  be small enough that the estimate~\eqref{eq:semismoothsets} holds for $x \in B_{\delta}(\bar x) \cap \cX$ and $y \in B_{2\delta}(\bar x) \cap \cY$. 
Let us first suppose that $z \in B_{\delta}(\bar x)\backslash \cX$ and let 
$$
v_z := \frac{z - \hat z}{\| z- \hat z\|} \in \partial \dist(z, \cX), \qquad \text{for some $\hat z \in P_{\cX}(z)$}.
$$
Recall that $v_z \in N_{\cX}(\hat z)$. Thus, since $\hat z \in B_{2\delta}(\bar x),$ we have
\begin{align*}
|\dist(z, \cX) + \dotp{v_z, y - z}| = |\dist(z, \cX) + \dotp{v_z, \hat z - z}| + |\dotp{v_z, y - \hat z}| &\leq C\|y - \hat z\|^{1+\eta}
\end{align*}
since $\dist(z, \cX) + \dotp{v_z, \hat z - z} = 0$. Now observe that 
$$
\|y - \hat z\|^{1+\eta} \leq 2\|y - z\|^{1+\eta} + 2\|z - \hat z\|^{1+\eta} \leq 4\|y - z\|^{1+\eta}.
$$
Consequently,  $(b)$-regularity with subgradient $v_z$ then follows from the bound:
$$
|\dist(z, \cX) + \dotp{v_z, y - z}| \leq 4C\|y - z\|^{1+\eta}.
$$
Since $\partial \dist(z, \cX) = \text{conv}\, \frac{z - P_{\cX}(z)}{\dist(z, \cX)}$, the $(b)$-regularity estimate with arbitrary $v\in \partial \dist(z, \cX)$ follows from averaging the above bound over all possible projections $\hat z$.
Next, let $z \in \cX \cap B_{\delta}(\bar x)$. Then $\partial \dist(z,\cX) = N_{\cX}(z) \cap \closedball$. Thus, the $(b)$-regularity estimate is precisely the estimate~\eqref{eq:semismoothsets}.

We now prove the Items. First note that Item~\ref{lem:bregdistance:semialgebraic} follows from the work~\cite{BDL09} applied to distance functions. Second, Item~\ref{lem:bregdistance:cone} follows from~\cite[Proposition 2.3.1]{DDJ21}. 
Finally, we prove Item~\ref{lem:bregdistance:manifold}. Suppose that $\cX = \cY$ and $ \cX$ is a $C^2$-smooth manifold around $\bar x$. 
Then there exists $C > 0$ such that $y - x \in T_{\cX}(x) + C\|y - x\|^2 \closedball$ for all $x, y \in \cX$ near $\bar x$. Thus, we have $
|\dotp{v, y - x} | \leq C\|v\|\|x - y\|^2  \text{ for all $x, y \in \cX$ near $\bar x$ and $v \in N_{\cX}(x)$,}
$
as desired.
\end{proof}

\subsection{Calculus}
Next we turn our attention to a few basic calculus results. The following theorem develops a chain rule for $(b)$-regularity. 
\begin{thm}[Chain rule]
  Consider two locally Lipschitz mappings $F_1\colon\RR^{d_1}\to\RR^{d_2}$ and $F_2\colon\RR^{d_2}\to\RR^{d_3}$, and define the composition $F = F_2 \circ F_1$.  Consider two locally bounded set-valued mappings $\jac_1 \colon\RR^{d_1} \multo \RR^{d_2 \times d_1}$ and $\jac_2 \colon\RR^{d_2} \multo \RR^{d_3 \times d_2}$ and
  define the composition 
  $$
  \jac(x):= \{ A_2 A_1 \colon A_1 \in \jac_1(x) \text{ and } A_2 \in \jac_2(F_1(x))\}  \qquad \text{ for all $x \in \RR^d$.}
  $$
Fix a set $\cY_2\subset \RR^{d_2}$,  define $\cY_1 :=F_1^{-1}(\cY_2)$, and let $\bar x \in \cY_1$. Suppose that 
\begin{enumerate}
\item  $(F_1, \jac_1)$ is $(b)$-regular along $\cY_1$ at $\bar x$ with exponent $1+\eta_1$.
\item $(F_2, \jac_2)$ is $(b)$-regular along $\cY_2$ at $F_1(\bar x)$ with exponent $1+\eta_2$.
\end{enumerate}
  Then $(F, \jac)$ is $(b)$-regular along $\cY_1$ at $\bar x$ with exponent $1+\min\{\eta_1, \eta_2\}$. 
\end{thm}
\begin{proof}
Let $U$ be a neighborhood of $F_1(\bar x)$ and $C > 0$ be a constant such that 
$$
\|F_2(z') - (F_2(z) +A_2(z'-z))\| \leq C\|z' - z\|^{1+\eta_2}
$$
for all $z \in U$, $A_2 \in \jac_2(z)$, and $z' \in U \cap \cY_2$.
Let $V = F_1^{-1}(U)$ and let $V' \subseteq V$ be a neighborhood of $\bar x$ small enough that there exists $\beta > 0$ with 
\begin{align*}
  \|F_1(y) - (F_1(x) + A_1 (y - x))\| \leq \beta\|x - y\|^{1+\eta_1}
\end{align*}
for all $x \in V'$, $A_1 \in \jac_1(x)$, and $y \in V' \cap \cY_1$.
Now, given $x \in V'$, select any $A_1 \in \jac_1(x)$ and $A_2 \in \jac_2(F_1(x))$. Let $L > 0$ satisfy $L \geq \sup_{A \in \jac_2(F_1(V'))} \opnorm{A}$. In addition, assume that $L$ is a Lipschitz constant for $F_1$ on $V'$. Then for all $x \in V'$ and $y \in V' \cap \cY_1$, we have
\begin{align*}
&|F(y) - (F(x) + A_2A_1(y - x))| \\
&\leq  |F_2(F_1(y)) - (F_2(F_1(x)) + A_2(F_1(y) - F_1(x)))| + \opnorm{A_2} \|F_1(y) - (F_1(x) + A_1 (y - x))\|\\
&\leq C\|F_1(y) - F_1(x)\|^{1+\eta_2} + \beta L\|x-y\|^{1+\eta_1}\\
&\leq CL^{1+\eta_2}\|y - x\|^{1+\eta_2} + \beta L\|x-y\|^{1+\eta_1},
\end{align*}
where the third inequality follows from the inclusions  $x \in V'$, $F_1(x) \in U$, $y \in V'\cap \cY_1$, and $F_1(y) \in U \cap \cY_2$. The proof then follows. 
\end{proof}

The chain rule immediately leads to leads to a sum-rule. The proof is routine, so we omit it.
\begin{corollary}[Sum Rule]\label{cor:bregsum}
Consider locally Lipschitz mappings $F_i \colon \RR^{d} \rightarrow \RR^{m}$, locally bounded set-valued mappings $\jac_i \colon \RR^d \multo \RR^{m \times d}$, and sets $\cY_i\subseteq \RR^{d}$ indexed by a finite set $I$. Define the set $\cY := \bigcap_{i \in I} \cY_i$, the mapping $F = \sum_{i\in I} F_i$, and the mapping
$$
\jac(x) = \left\{\sum_{i\in I} A_i \colon A_i \in \jac_i(x) \text{ for } i \in I \right\} \qquad \text{for all $x \in \RR^d$.}
$$
Suppose that for each $i \in I$, the pair $(F_i, \jac_i)$ is $(b)$-regular along $\cY_i$ at $\bar x$ with exponent $1+\eta_i$. Then
 $(F, \jac)$ is $(b)$-regular along $\cY$ at $\bar x$ with exponent $1+\min_{i\in I}\{\eta_i\}$. 
\end{corollary}

The final result of this section states that $(b)$-regularity is preserved by ``stacking" mappings. The proof is straightforward, so we omit it.
\begin{lemma}[Stacking]
Consider locally Lipschitz mappings $F_1 \colon \RR^{d} \rightarrow \RR^{d_1}$ and $F_2 \colon \RR^{d} \rightarrow \RR^{d_2}$ and define the mapping $F(x) := (F_1(x), F_2(x))$ for all $x \in \RR^d$. Consider two locally bounded set-valued mappings $\jac_1 \colon \RR^d \multo \RR^{d_1 \times d}$ and $\jac_2 \colon \RR^d \multo \RR^{d_2 \times d}$ and define the mapping 
$$
\jac(x) = \left\{\begin{bmatrix} A_1 \\ A_2 \end{bmatrix} \colon A_1 \in \jac_1(x)  \text{ and }  A_2 \in \jac_2(x)\right\} \qquad \text{for all $x \in \RR^d$.}
$$
Fix sets $\cY_1\subseteq \RR^{d}$ and $\cY_2 \subseteq \RR^d$ and define $\cY := \cY_1\cap \cY_2$. Suppose that for $i \in \{1, 2\}$ the pair $(F_i, \jac_i)$ is $(b)$-regular along $\cY_i$ at $\bar x$ with exponent $1+\eta_i$.
Then $(F, \jac)$ is $(b)$-regular along $\cY$ at $\bar x$ with exponent $1+\min\{\eta_1, \eta_2\}$. 
\end{lemma}

To close this section, we mention that further calculus rules (e.g., preservation under spectral lifts) may be adapted from those in~\cite{DDJ21}.

\subsection{Consequences for semialgebraic mappings}
Given the chain rule and the basic examples of Lemma~\ref{lem:basicexamples}, we have the following immediate consequences for semialgebraic mappings.
\begin{corollary}[Chain-rule with semialgebraic mappings]\label{cor:semib}
Suppose that $F_1 \colon \RR^{d_1} \rightarrow \RR^{d_2}$ and $F_2 \colon \RR^{d_2} \rightarrow \RR^{d_3}$ are locally Lipschitz mappings. Let $\bar x \in \RR^{d_1}$ and define the set $\cY := F_1^{-1}(F_1(\bar x))$. Suppose that 
\begin{enumerate}
\item the mapping $F_2$ is semialgebraic and $\jac_2 \colon \RR^{d_2} \multo \RR^{d_3 \times d_2}$ is a semialgebraic conservative set-valued vector field for $F_2$ (e.g., $\partial F_2$).
\item and that either of the following hold:
\begin{enumerate}
\item near $\bar x$ the mapping $F_1$ is $C^1$ and the Jacobian $\jac_1 := \nabla F_1$ is Lipschitz.
\item there exists a locally bounded set-valued mapping $\jac_1 \colon \RR^{d_1} \multo \RR^{d_2\times d_1}$ such that $(F_1, \jac_1)$ is $(b)$-regular along $\cY$ at $\bar x$ with exponent $1 + \eta_1$.
\end{enumerate}
\end{enumerate}
Define $\jac(x):= \{ A_2 A_1 \colon A_1 \in \jac_1(x) \text{ and } A_2 \in \jac_2(F_1(x))\}$ for all $x \in \RR^d$.
Then there exists $\eta > 0$ such that $(F_2 \circ F_1, \jac)$ is $(b)$-regular along $\cY$ at $\bar x$ with exponent $1+\eta$.
\end{corollary}
As in~\cref{cor:bregsum}, the result of Corollary~\ref{cor:semib} leads to a ``sum rule" for semialgebraic mappings, whose details are immediate.

\subsection{Consequences for functions}\label{sec:bregfunctions}

In this section, we prove the claims of Section~\ref{sec:examplesofbregintro}.

\subsubsection{Proof of Proposition~\ref{prop:basicsemialgebraic}.}

The result follows from Corollary~\ref{cor:semib} with mappings $F_2 := h$ and  $F_1 := F(x)$.

\subsubsection{Proof of Proposition~\ref{prop:rootfinding}.}

Define the mapping $F = F_2 \circ F_1$ and $\jac(x) := \{ A_2 \nabla F_1(x) \colon  \text{ and } A_2 \in \partial F_2(F_1(x))\}$.
By Corollary~\ref{cor:semib}, the pair $(F,\jac)$ is $(b)$-regular along $\cY$ at $\bar x$ with exponent $1+\eta$.
Applying Corollary~\ref{cor:semib} again to the composition $\| F(x)\|$ gives the result.

\subsubsection{Proof of Proposition~\ref{prop:semialgebraicfeasibility}.}

Recall that $g_i(x) \in \partial \dist(x, \cX_i)$. Consequently, by Lemma~\ref{lem:basicexamples} there exists $\eta_i > 0$ such that the semialgebraic mappings $(\dist(x, \cX_i), g_i^\T)$ are $(b)$-regular along $\{\bar x\}$ at $\bar x$ with exponent $1+\eta_i$ for $i \in I$. Therefore, the result follows by Corollary~\ref{cor:bregsum}, as desired.

\section{Proofs of the main algorithmic results}\label{sec:proofofalgsresults}

Throughout this section, we assume that \ref{item:assumption:main:sharpness} and~\ref{item:assumption:main:b} are in force.
We begin with a few lemmata that will reappear in several proofs.
The first Lemma ensures we can use
the $(b)$-regularity estimate with $y = P_{\cX_\ast}(x)$. The proof is straightforward, so we
omit it.
\begin{lemma}\label{lem:distanceincrease}
Let $\delta' > 0$ and let $y \in \cX_{\ast}$. We have
\begin{align*}
x\in B_{\delta'}(y) \text{ and }  \hat x \in P_{\cX_\ast}(x)  \implies \hat x\in B_{2\delta'}(y).
\end{align*}
\end{lemma}

The next property is fundamental to the convergence of the $\polyak$ and $\bundle$ procedures.
It states that negative subgradients aim towards the set of minimizers.
\begin{lemma}[Aiming]
  \label{lemma:aiming-inequality}
  Fix a point $x$ satisfying the bound $\norm{x - \bar{x}} \leq \delta$, as well as $\dist(x, \cX_{\ast})
  \leq (\frac{\mu}{2\constb})^{1/\eta}$.
  Then
  \begin{align}\label{eq:aiming}
    \dotp{v, x - \hat x} \geq \frac{\mu}{2}\cdot \dist(x, \cX_\ast) \qquad \text{for all $v \in \subg(x)$ and $\hat x \in P_{\cX_\ast}(x)$.}
  \end{align}
In particular, if $x \notin \cX_\ast$, the bound $\|v\| \geq \mu/2$ holds for all $v \in \subg(x)$.
\end{lemma}
\begin{proof}
By $(b)$-regularity and sharp growth, we have
$$
\constb\|x - \hat x \|^{1+\eta} + \dotp{v, x - \hat x} \geq f(x) - f(\hat x) \geq \mu \, \|x - \hat x\|.
$$
Rearranging gives~\eqref{eq:aiming}. The lower bound $\|v\| \geq \mu/2$ is immediate.
\end{proof}

Finally we derive the main consequence of~\ref{item:assumption:main:fallback} that is used in this work. The following lemma shows that when iterated, algorithmic mappings generate iterates with two desirable properties: they do not travel far from $\bar x$ and they linearly converge to $\cX_\ast$.
\begin{lemma}
  \label[lemma]{lemma:fallback-stay-in-ball}
  Let $\cA$ satisfy~\ref{item:assumption:main:fallback}. Fix a point $x \in \RR^d$ satisfying
  \[
    \norm{x - \bar{x}} < \frac{(1-\rho)\radius_1}{2} \qquad \text{ and } \qquad
    \dist(x, \cX_{\ast}) < \radius_2.
  \]
  Define $z_{-1} := x$ and for all $\inner \geq 0$, define $z_{\inner} :=
  \cA^{\circ \inner}(x)$. Then for all $\inner \geq 0$ and $\hat z_{i} \in P_{\cX_\ast}(z_{i})$, we have
  \[
    \|z_{\inner} - \bar x\| < \radius_1 \qquad \text{ and } \qquad
    \dist(z_{\inner}, \cX_{\ast}) \leq \|z_\inner - \hat z_{\inner-1}\| \leq \rho^{\inner} \, \dist(z_0, \cX_{\ast}).
  \]
\end{lemma}
\begin{proof}
  Assume without loss of generality that $f^{\ast} = 0$ and define $z_{-1} = z_{0}$.
  We show that the following holds for all $\inner \ge 0$:
  \begin{subequations}
    \begin{align}
      \norm{z_{\inner} - \bar{x}} &\leq \norm{z_0 - \bar{x}} \left(
        1 + (1 + \rho) \sum_{j=0}^{\inner-1} \rho^j \right) < \radius_1,
        \label{eq:fallback-near-xbar} \\
      \norm{z_{\inner} - \hat{z}_{\inner-1}} &\leq \rho^{\inner} \norm{z_0 - \hat{z}_0}.
      \label{eq:fallback-dist-reduction}
    \end{align}
  \end{subequations}
  The base case follows trivially. Now, assume the bounds hold up to some index $\inner$.
  Then,~\eqref{eq:fallback-dist-reduction} ensures $\dist(z_i, \cX_{\ast}) \leq
  \dist(z_0, \cX_{\ast}) < \radius_2$. Therefore, by~\ref{item:assumption:main:fallback},
  we have
  \[
    \norm{z_{\inner+1} - \hat{z}_{\inner}} \leq \rho \norm{z_{\inner} - \hat{z}_{\inner}}
    \leq \rho^{\inner+1} \, \norm{z_0 - \hat{z}_0},
  \]
  which proves~\eqref{eq:fallback-dist-reduction}. Finally, we have
  \begin{align*}
    \norm{z_{\inner+1} - \bar{x}} &\leq \norm{z_{\inner+1} - \hat{z}_{\inner}} +
      \norm{z_{\inner} - \hat{z}_{\inner}} + \norm{z_{\inner} - \bar{x}} \\
                                  &\leq
      (1 + \rho)\norm{z_{\inner} - \hat{z}_{\inner}} +
      \norm{z_0 - \bar{x}} \left(1 + (1 + \rho) \sum_{j=0}^{\inner-1} \rho^j\right) \\
                                  &\leq
      (1 + \rho) \rho^{\inner} \norm{z_0 - \hat{z}_{0}} +
      \norm{z_0 - \bar{x}} \left(1 + (1 + \rho) \sum_{j=0}^{\inner-1} \rho^j\right) \\
                                  &\leq
      \norm{z_0 - \bar{x}} \left(1 + (1 + \rho) \sum_{j=0}^{\inner} \rho^j\right) < \radius_1,
  \end{align*}
  where the penultimate inequality follows from the bound $\norm{z_0 - \hat{z}_0}$ and
  the last inequality follows from $\norm{z_0 - \bar{x}} \leq (1 - \rho) \radius_1 / 2$.
  This proves~\eqref{eq:fallback-near-xbar} and completes the proof.
\end{proof}

\subsection{Proof of~\cref{thm:polyak}}\label{sec:thm:polyak}

The following lemma proves Theorem~\ref{thm:polyak}.
\begin{lemma}[One step improvement]
\label{lemma:one-step-improvement}
Let $L$ be an upper bound for the maximal norm element of
$\subg(B_{\delta}(\bar x))$.
Let $s \colon \RR^d
  \rightarrow \RR^d$ satisfy $s(x) \in \subg(x)$ for all $x \in \RR^d$. Define
  the mapping
  \[
    \cA(x) := \begin{cases}
    x - \frac{f(x) - f^{\ast}}{\norm{s(x)}^2} s(x) & \text{if $s(x) \neq  0$;}\\
    x &\text{otherwise.}
    \end{cases}
  \]
  Then $\cA$ satisfies~\ref{item:assumption:main:fallback} with
  \[
    \radius_1 = \delta; \quad
    \radius_2 = \left(\frac{\mu}{4 \constb}\right)^{1/\eta};
    \quad \rho = \sqrt{1 - \frac{\mu^2}{4L^2}}.
  \]
Consequently, Theorem~\ref{thm:polyak} holds.
\end{lemma}
\begin{proof}
Assume without loss of generality that $f^\ast = 0$.
Fix a point $x \notin \cX$ satisfying  the bounds $\norm{x - \bar{x}} < \radius_1$ and $
      \dist(x, \cX_{\ast}) < \radius_2$.
      Notice that Lemma~\ref{lemma:aiming-inequality} guarantees that $s(x) \neq 0$.
Choose $\hat x \in P_{\cX_\ast}(x)$ and observe that
 \begin{align*}
    \norm{\cA(x) - \hat{x}}^2 &=
    \norm{x - \hat{x}}^2 + \frac{f(x)^2}{\norm{v}^2} -
    2 \frac{f(x)}{\norm{v}^2} \ip{v, x - \hat{x}} \\
                           &=
    \norm{x - \hat{x}}^2 + \frac{f(x)}{\norm{v}^2}
    \left(f(x) + \ip{v, \hat{x} - x}\right)
    - \frac{f(x)}{\norm{v}^2} \ip{v, x - \hat{x}} \\
                           &\leq
    \norm{x - \bar{x}}^2 + \frac{f(x)}{\norm{v}^2}
    \left( \constb \norm{x - \hat{x}}^{1 + \eta} - \frac{\mu}{2} \norm{x - \hat{x}}\right) \\
                           &\leq
    \norm{x - \hat{x}}^2 - \frac{f(x)}{\norm{v}^2} \cdot \frac{\mu}{4} \norm{x - \hat{x}} \\
                           &\leq
                           \norm{x - \hat{x}}^2\left(1 - \frac{\mu^2}{4\|v\|^2} \right).
  \end{align*}
 where the first inequality follows from $(b)$-regularity and~\cref{lemma:aiming-inequality},
 the second inequality follows from the assumed bound for $\dist(x, \cX_{\ast})$, and
 the last inequality follows from sharpness.
 Thus, $\cA$ satisfies~\ref{item:assumption:main:fallback}. Consequently, by Lemma~\ref{lemma:fallback-stay-in-ball}, Theorem~\ref{thm:polyak} holds.
\end{proof}

\subsection{Proof of Theorem~\ref{corollary:function-value-reduction}}\label{proof:corollary:function-value-reduction}

We first establish some assumptions and notation.
Without loss of generality we assume $f^\ast = 0$.
We let $\{y_\inner\}_{\inner}$ and $\{v_\inner\}_{\inner}$ denote the iterates and generalized gradients generated by Algorithm~\ref{alg:build-bundle-method}.
We let
$$
\hat y_{\inner} \in P_{\cX_\ast}(y_{\inner}) \qquad \text{for $\inner = 0, \ldots, d$}
$$
denote projections of the iterates onto $\cX_\ast$. Note that whenever $y_{\inner}
\in B_{\delta/2}(\bar x)$, we have $\hat y_{\inner} \in B_{\delta}(\bar x)$ (see Lemma~\ref{lem:distanceincrease}).

We now turn to several technical Lemmas and Propositions.
The first proposition is proved in Appendix~\ref{sec:build-bundle-missing-proofs}.
It will help us ensure that $\bundle$ terminates after at most $d$ iterations.% since the rank of the matrix $A_k$ grows with $k$.
\begin{proposition}
  \label[proposition]{proposition:build-bundle-sval-lower-bound}
  Fix $\inner \geq 1$ and suppose that $\norm{P_{\ker(A_j)} v_j} \ge \alpha > 0$ for all $j \le \inner$.
  Then the following holds:
  \begin{equation*}
    \rank(A_{\inner+1}) = \inner+1 \quad
    \text{and} \quad
    \sigma_{\inner+1}(A_{\inner+1}) \geq
    \min\set{\opnorm{A_0}, 1} \left(\frac{\alpha}{L \sqrt{2}}\right)^{\inner}.
  \end{equation*}
\end{proposition}
We proceed with several technical Lemmas, whose proofs appear inline. First we show the
following decomposition of $y_\inner - \hat{y}_0$, which we use repeatedly in the below.
\begin{lemma}
  \label[lemma]{lemma:build-bundle-error-decomposition}
  For all $i \geq 1$, the following identity holds:
  \begin{equation*}
    y_{\inner} - \hat{y}_0 = \proj_{\ker(A_{\inner})}(y_0 - \hat{y}_0)
    - A_{\inner}^{\dag} \bmx{f(y_j) + \ip{v_j, \hat{y}_0 - y_j}}_{j=0}^\inner.
     \end{equation*}
\end{lemma}
\begin{proof}
  Recall the projection formula
  $I - A_{\inner}^{\dag} A_{\inner} = \proj_{\ker(A_{\inner})}$. The claimed decomposition
  follows since
  \begin{align*}
    y_{\inner} - \hat{y}_0 &= \proj_{\ker(A_{\inner})}(y_0 - \hat{y}_0) +
    A_{\inner}^{\dag} A_{\inner} (y_0 - \hat{y}_0)
    - A_{\inner}^{\dag} \bmx{f(y_j) + \ip{v_j, y_0 - y_j}}_{i=0}^{\inner} \\
                          &= \proj_{\ker(A_{\inner})}(y_0 - \hat{y}_0)
    - A_{\inner}^{\dag} \bmx{f(y_j) + \ip{v_j, \hat{y}_0 - y_j}}_{i=0}^{\inner},
  \end{align*}
  as desired.
\end{proof}
The second lemma shows that the update $y_{\inner}$ improves upon $y_0$
whenever the gap vector $y_{\inner} - \hat y_0$ has a large component in $\ker(A_{\inner})$.
\begin{lemma}[Distance reduction]
  \label[lemma]{lemma:alternative-I-manifold1}
  Fix $\gamma \leq \mu/2L$. Suppose that for some $i \geq 1$, we have
  $$
  \norm{\proj_{\ker(A_{\inner})^\perp}(y_{\inner} - \hat y_0)} \leq \gamma\norm{y_{i} -\hat y_i}.
  $$
  Then we have the following bound
  \begin{equation*}
    \norm{y_{\inner} - \hat{y}_0} \leq \sqrt{\frac{1 - \frac{\mu^2}{4L^2}}{1 - \gamma^2}}
    \norm{y_0 - \hat{y}_0} \leq
    \norm{y_0 - \hat{y}_0}.
  \end{equation*}
\end{lemma}
\begin{proof}
By definition, we have $v_{0} \perp \ker(A_{\inner})$.
In addition, by \cref{lemma:aiming-inequality}, we have the bound  $\abs{\dotp*{\frac{v_0}{\|v_0\|}, y_{0} - \hat{y}_0}} \geq
  \frac{\mu}{2L}\, \norm{y_{0} - \hat{y}_0}$. Taken together, these imply
\begin{align}\label{eq:y0good}
  \norm{\proj_{\ker(A_{\inner})^{\perp}}(y_0 - \hat{y}_0)}^2 \geq
  \frac{\mu^2}{4L^2} \norm{y_0 - \hat{y}_0}^2.
\end{align}
  Next, observe that by Lemma~\ref{lemma:build-bundle-error-decomposition}, we have
  $P_{\ker(A_{\inner})}(y_{\inner} - \hat y_0) = P_{\ker(A_{\inner})}(y_0 - \hat y_0)$.
  Consequently, we have
  \begin{align*}
    \norm{y_{\inner} - \hat{y}_0}^2 &= \norm{\proj_{\ker(A_{\inner})}(y_0 - \hat{y}_0)}^2
    + \norm{\proj_{\ker(A_{\inner})^{\perp}}(y_{\inner} - \hat{y}_0)}^2 \\
                             &=
    \norm{y_0 - \hat{y}_0}^2 - \norm{\proj_{\ker(A_{\inner})^{\perp}}(y_0 - \hat{y}_0)}^2
    + \norm{\proj_{\ker(A_{\inner})^{\perp}}(y_{\inner} - \hat{y}_0)}^2 \\
                             &\leq
    \norm{y_0 - \hat{y}_0}^2 \left(1 - \frac{\mu^2}{4L^2}\right) +
    \norm{\proj_{\ker(A_{\inner})^{\perp}}(y_{\inner} - \hat{y}_0)}^2 \\
                             &\leq
    \norm{y_0 - \hat{y}_0}^2 \left(1 - \frac{\mu^2}{4L^2}\right) +
    \gamma^2 \norm{y_i - \hat{y}_0}^{2},
  \end{align*}
  where the penultimate inequality follows from~\eqref{eq:y0good} and the last
  inequality follows from $ \norm{\proj_{\ker(A_{\inner})^\perp}(y_{\inner} - \hat y_0)} \leq \gamma\norm{y_{i} -\hat y_i}$ and the bound $\|y_{i}- \hat y_i\| \leq \|y_i - \hat y_0\|$. Rearranging, we arrive at the desired conclusion.
\end{proof}
The third and final lemma is the core of our argument. It shows that
$\bundle$ increases the rank of $A_{\inner}$ until it finds a vector $y_{\inner}$ that superlinearly improves upon $y_0$.
\begin{lemma}[Alternatives]
\label[lemma]{lemma:manifold-alternatives}
Fix $\gamma \leq \mu/4L$ and suppose that for all $j \leq \inner$, we have
\begin{align}
\begin{aligned}
\|y_j - \hat y_0\| &\leq \|y_0 - \hat y_0\|;\\
\|\proj_{\ker(A_j)^\perp}(y_j - \hat y_0)\| &\leq \gamma \norm{y_{j} - \hat{y}_j}; \\
\dotp{v_j, \proj_{\ker(A_j)}(y_j - \hat{y}_j)} &\geq \frac{\mu}{8} \norm{y_j - \hat{y}_j};
\end{aligned}\label{eq:alternativesassumption}
\end{align}
and the inclusions $y_0 \in B_{\frac{\delta}{4}}(\bar{x})$, $y_j \in B_{\frac{3 \delta}{4}}(\bar x)$.
Then
\begin{enumerate}
\item \label{item:rankupdate} $\rank(A_{\inner+1}) = \inner+1$ and
  \[
    \opnorm{A_{\inner+1}^{\dag}} \leq
    \max\set{1, \frac{2}{\mu}} \left( \frac{8\sqrt{2} L}{\mu} \right)^{\inner}.
  \]
\item At least one of the following hold:
\begin{enumerate}
\item \label{item:maintainprogress} {\bf (Maintain progress)} We have the inequalities
  \begin{align}
      \norm{y_{\inner+1} - \hat{y}_0} &\leq \norm{y_0 - \hat{y}_0}; \notag\\
    \norm{\proj_{\ker(A_{\inner+1})^\perp}(y_{\inner+1} - \hat{y}_0)} &\leq \gamma \norm{y_{\inner+1} - \hat{y}_{\inner + 1}};\notag \\
    \ip{v_{\inner+1}, P_{\ker(A_{\inner+1})}({y}_{\inner+1} - \hat{y}_{\inner + 1})} &\geq
    \frac{\mu}{8} \norm{y_{\inner+1} - \hat{y}_{\inner + 1}}; \label{eq:dplowerbound}
  \end{align}
  and the inclusion $y_{\inner+1} \in B_{\frac{3 \delta}{4}}(\bar{x})$.
  \label{item:maintain-progress}
\item {\bf (Superlinear improvement I)}: the next iterate satisfies
  \[
    \begin{aligned}
      \dist(y_{\inner+1}, \cX_{\ast}) &\leq
        \frac{8L}{\mu} \left(4^{1+\eta} \constb \sqrt{\inner+1}\opnorm{A_{\inner+1}^{\dag}}\right)
      \norm{y_{0} - \hat{y}_0}^{1+\eta}; \quad \\
        \norm{y_{\inner+1} - \hat{y}_0} &\leq
        \norm{y_0 - \hat{y}_0}.
    \end{aligned}
  \]
  \label{item:local-quadratic-convergence}
\item {\bf (Superlinear improvement II)}: the next iterate satisfies
  \begin{align*}
    \dist(y_{\inner+1}, \cX_\ast) &\leq  \frac{\constb \sqrt{\inner+1}}{\gamma}\opnorm{A_{\inner+1}^{\dagger}}
    \norm{y_0 - \hat y_0}^{1+\eta}; \\
    \|y_{i+1} - \hat y_0\| &\leq \|y_0 - \hat y_0\| +  \constb \sqrt{\inner + 1} \opnorm{A_{\inner+1}^{\dag}} \norm{y_0 - \hat{y}_0}^{1 + \eta}.
  \end{align*}
  \label{item:local-superlinear-convergence}
\end{enumerate}
\end{enumerate}
\end{lemma}
\begin{proof}
We first prove Item~\ref{item:rankupdate}. Observe that
$y_0$ satisfies the conditions of~\cref{lemma:aiming-inequality} by assumption, so
$\ip{v_0, y_0 - \hat{y}_0} \geq \frac{\mu}{2} \norm{y_0 - \hat{y}_0}$. Consequently, since $A_0 = [v_0^{\T}]$, we have $\opnorm{A_0} \geq \frac{\mu}{2}$.
Furthermore, the third inequality of~\eqref{eq:alternativesassumption} implies that $\norm{\proj_{\ker(A_j)}(v_j)} \geq \frac{\mu}{8}$ for all $j \leq
\inner$. Thus, \cref{proposition:build-bundle-sval-lower-bound} yields
\[
  \rank(A_{\inner+1}) = \inner + 1, \quad
  \text{and} \quad
  \sigma_{\min}(A_{\inner+1}) = \sigma_{\inner+1}(A_{\inner+1}) \geq
  \min\set{\frac{\mu}{2}, 1} \cdot \left( \frac{\mu}{8\sqrt{2} L} \right)^{\inner}.
\]
The inequality follows after noticing that $\opnorm{A^{\dag}} = 1 / \sigma_{\min}(A)$.

\vspace{12pt}

\noindent For the rest of the proof, we perform a case-by-case analysis.

\paragraph{Case 1:} Suppose $\norm{\proj_{\ker(A_{\inner+1})^{\perp}}(y_{\inner+1} - \hat {y}_{0})}
\geq \gamma \norm{y_{\inner + 1} - \hat y_{\inner + 1}}$. When this holds, we have
\begin{align*}
\gamma \|y_{\inner+1} - \hat y_{\inner + 1}\| &\leq
  \norm{\proj_{\ker(A_{\inner+1})^{\perp}}(y_{\inner+1} - \hat{y}_0)}\\
  &=
  \norm{A_{\inner+1}^{\dag}\bmx{f(y_j) + \dotp{v_j, \hat{y}_0 - y_j}}_{j=0}^{\inner}}
  \\
    &\leq \opnorm{A_{\inner+1}^{\dag}} \sqrt{\sum_{j=0}^{\inner} \constb^2 \norm{\hat{y}_0 - y_j}^{2(1 + \eta)}}\\
    &\leq
  \constb \sqrt{\inner + 1} \opnorm{A_{\inner+1}^{\dag}} \norm{y_0 - \hat{y}_0}^{1 + \eta}
\end{align*}
where the first inequality follows by assumption, the second inequality
follows by $(b)$-regularity, and the third inequality follows from the assumption
that $\norm{y_j - \hat{y}_0} \leq \norm{y_0 - \hat{y}_0}$ for all
$j \leq \inner$. In addition, using Lemma~\ref{lemma:build-bundle-error-decomposition} and the above inequality, we find that
\begin{align*}
\|y_i - \hat y_0\| &\leq \|y_0 - \hat y_0\| +  \norm{A_{\inner+1}^{\dag}\bmx{f(y_j) + \dotp{v_j, \hat{y}_0 - y_j}}_{j=0}^{\inner}} \\
&\leq \|y_0 - \hat y_0\| +  \constb \sqrt{\inner + 1} \opnorm{A_{\inner+1}^{\dag}} \norm{y_0 - \hat{y}_0}^{1 + \eta},
\end{align*}
as desired.
\;

\paragraph{Case 2:} Suppose $\|\proj_{\ker(A_{\inner+1})^\perp}(y_{\inner+1} - \hat{y}_0)\| \leq  \gamma \norm{y_{\inner + 1} - \hat y_{\inner + 1}}$. Under this condition, Lemma \ref{lemma:alternative-I-manifold1}  ensures that
\begin{equation}
  \norm{y_{\inner+1} - \hat{y}_0} \leq \norm{y_0 - \hat{y}_0}.
  \label{eq:maintain-progress-distance}
\end{equation}
This proves the first two inequalities of~\cref{item:maintain-progress}.
To prove the inclusion $y_{\inner+1} \in B_{\frac{3\delta}{4}}(\bar x)$, note
that since $y_0 \in B_{\frac{\delta}{4}}(\bar{x})$ and
$\norm{y_0 - \hat{y}_0} \leq \norm{y_0 - \bar{x}}$, it follows that
\begin{align*}
  \norm{y_{\inner+1} - \bar{x}} &\leq
  \norm{y_{\inner+1} - \hat{y}_0} + \norm{\hat{y}_0 - y_0}
  + \norm{y_0 - \bar x}
                            \leq
  2 \norm{y_0 - \hat{y}_0} + \norm{y_0 - \bar x} < \frac{3 \delta}{4},
\end{align*}
as desired.

In the remainder of the proof, we show that either we obtain local superlinear improvement or the lower bound~\eqref{eq:dplowerbound} holds.
To that end, we first note that
\begin{align}
  \frac{\mu}{2} \norm{y_{\inner+1} - \hat{y}_{\inner+1}} &\leq
  \dotp{v_{\inner+1}, y_{\inner+1} - \hat{y}_{\inner+1}} \notag \\
                                                         &=
  \dotp{v_{\inner+1}, P_{\ker(A_{\inner+1})^{\perp}}(y_{\inner+1} - \hat{y}_{\inner+1})}
  +
  \dotp{v_{\inner+1}, P_{\ker(A_{\inner+1})}(y_{\inner+1} - \hat{y}_{\inner+1})} \notag \\
                                                         &\leq
  L \norm{P_{\ker(A_{\inner+1})^{\perp}}(y_{\inner+1} - \hat{y}_{\inner+1})}
  + \dotp{v_{\inner+1}, P_{\ker(A_{\inner+1})}(y_{\inner+1} - \hat{y}_{\inner+1})},
  \label{eq:case-2-first}
\end{align}
where the first inequality follows from the Lemma~\ref{lemma:aiming-inequality} and the third inequality follows from Cauchy-Schwarz. We now upper bound the first term in the right-hand side of~\eqref{eq:case-2-first}.
\begin{claim}
The following bound holds:
\begin{align*}
  \norm{P_{\ker(A_{\inner+1})^{\perp}}(y_{\inner+1} - \hat{y}_{\inner+1})} &\leq
  (\mu/4L)   \norm{y_{\inner + 1} - \hat y_{\inner + 1}}
  + 4^{1+\eta} \constb \sqrt{\inner + 1} \opnorm{A_{\inner+1}^{\dag}}
  \norm{y_0 - \hat{y}_0}^{1+\eta}.% \\
\end{align*}
\end{claim}
\begin{proof}
First note that
\begin{align}
  \norm{P_{\ker(A_{\inner+1})^{\perp}}(y_{\inner+1} - \hat{y}_{\inner+1})} &\leq
  \norm{P_{\ker(A_{\inner+1})^{\perp}}(y_{\inner+1} - \hat{y}_{0})} +
  \norm{P_{\ker(A_{\inner+1})^{\perp}}(\hat{y}_{\inner+1} - \hat{y}_{0})} \notag \\
                                                                           &\leq
  \gamma   \norm{y_{\inner + 1} - \hat y_{\inner + 1}} +
  \norm{A_{\inner+1}^{\dag} A_{\inner+1}(\hat{y}_{\inner+1} - \hat{y}_0)},
  \label{eq:case-2-second}
\end{align}
where the second inequality follows from the assumption of this case and the projection identity
$P_{\ker(A_{i+1})^{\perp}} = A_{i+1}^{\dag} A_{i+1}$.
We now upper bound $\norm{A_{\inner+1}^{\dag} A_{\inner+1}(\hat{y}_{\inner+1} - \hat{y}_0)}$ in~\eqref{eq:case-2-second}. Indeed, $(b)$-regularity yields
\begin{equation*}
  \abs{f(y_j) + \dotp{v_j, \hat{y} - y_j}} \leq \constb \norm{\hat{y} - y_j}^{1 + \eta}
  \qquad \text{for all  $\hat{y} \in B_{2 \delta}(\bar{x})$ and $j \leq i$.}
\end{equation*}
 Consequently, for all $j \leq i$, we have
\begin{align*}
  \abs{\dotp{v_j, \hat{y}_{\inner+1} - \hat{y}_0}} &=
  \abs{f(y_j) + \dotp{v_j, \hat{y}_{\inner+1} - y_j} -
  \left( f(y_j) + \dotp{v_j, \hat{y}_0 - y_j} \right)} \\
                                                   &\leq
  \constb \left(\norm{\hat{y}_0 - y_j}^{1+\eta} + \norm{\hat{y}_{\inner+1} - y_j}^{1+\eta}\right) \\
                                                   &\leq
  \constb \left(\norm{y_j - \hat{y}_0}^{1+\eta} +
    (\norm{\hat{y}_{\inner+1} - y_{\inner+1}} + \norm{y_{\inner+1} - \hat{y}_0} +
     \norm{y_j - \hat{y}_0})^{1 + \eta} \right) \\
                                                   &\leq
  \constb \left(\norm{y_0 - \hat{y}_0}^{1 + \eta} +
  (2 \norm{y_{i+1} - \hat{y}_0} + \norm{y_0 - \hat{y}_0})^{1+\eta}\right) \\
                                                   &\leq
                                                   \constb 4^{1 + \eta} \norm{y_0 - \hat{y}_0}^{1+\eta},
\end{align*}
where the third inequality follows from~\eqref{eq:alternativesassumption} and
the fourth inequality follows from~\eqref{eq:maintain-progress-distance}.
Now, since $A_{\inner+1} = \bmx{v_1 & \dots & v_{\inner}}^{\T}$, it follows
that
\begin{align*}
  \norm{A_{\inner+1}^{\dag} A_{\inner+1}(\hat{y}_{\inner+1} - \hat{y}_0)} &=
  \opnorm{A_{\inner+1}^{\dag}}\sqrt{\sum_{j=0}^{\inner} \dotp{v_{j}, \hat{y}_{\inner+1} - \hat{y}_0}^2} \leq
  4^{1 + \eta} \constb \sqrt{\inner+1}\opnorm{A_{\inner+1}^{\dag}} \norm{y_0 - \hat{y}_0}^{1+\eta},
\end{align*}
Returning to~\eqref{eq:case-2-second}, we thus arrive at the bound
\begin{align*}
  \norm{P_{\ker(A_{\inner+1})^{\perp}}(y_{\inner+1} - \hat{y}_{\inner+1})} &\leq
  \gamma   \norm{y_{\inner + 1} - \hat y_{\inner + 1}}
  + 4^{1+\eta} \constb \sqrt{\inner + 1} \opnorm{A_{\inner+1}^{\dag}}
  \norm{y_0 - \hat{y}_0}^{1+\eta}. % \\
\end{align*}
Noting that $\gamma L \leq \mu/4$ yields the result.
\end{proof}

Therefore, plugging the conclusion of the claim into~\eqref{eq:case-2-first}, we obtain
\begin{align}
  \frac{\mu}{4} \norm{y_{\inner+1} - \hat{y}_{\inner+1}} &\leq
  C \norm{y_0 - \hat{y}_0}^{1+\eta} +
  \dotp{v_{\inner+1}, P_{\ker(A_{\inner+1})}(y_{\inner+1} - \hat{y}_{\inner+1})}
  \label{eq:case-2-omg}
\end{align}
for the constant
$  C := 4^{1+\eta}L \constb \sqrt{\inner+1} \opnorm{A_{\inner+1}^{\dag}}.$
We now analyze~\eqref{eq:case-2-omg} in two scenarios:
    First suppose that $C \norm{y_0 - \hat{y}_0}^{1+\eta} \leq \frac{\mu}{8} \norm{y_{\inner+1} -\hat{y}_{\inner+1}}$. Then
    upper bounding~\eqref{eq:case-2-omg} and rearranging yields
    \[
      \frac{\mu}{8} \norm{y_{\inner+1} - \hat{y}_{\inner+1}} \leq
      \dotp{v_{\inner+1}, P_{\ker(A_{\inner+1})}(y_{\inner+1} - \hat{y}_{\inner+1})},
    \]
    which proves~\eqref{eq:dplowerbound}. Thus, the conclusion of Item~\ref{item:maintainprogress} follows.
  Otherwise, the conclusion of~\cref{item:local-quadratic-convergence} follows by
    \[
      \norm{y_{\inner+1} - \hat{y}_{\inner+1}} \leq
      \frac{8C}{\mu} \norm{y_0 - \hat{y}_0}^{1+\eta},
    \]
    and equation~\eqref{eq:maintain-progress-distance}.
This completes the proof of the lemma.
\end{proof}

  We now complete the proof of the theorem. Let $y_{\inner}$ be the first iterate
such that Item~\ref{item:maintainprogress} in Lemma~\ref{lemma:manifold-alternatives}
does not hold (such an iterate must exist since the rank of $A_{\inner}$ increases
at each iteration). We first show that $\tilde x$ exists and $f(\tilde x) \leq f(y_{\inner})$.
Indeed, by~\cref{item:local-quadratic-convergence,item:local-superlinear-convergence}
there exists a constant $B>0$ such that
\begin{align}\label{eq:manifolddecrease}
 \|y_\inner - \hat y_\inner\| &\leq B\norm{y_0 - \hat{y}_0}^{1+\eta}; \notag\\
  \norm{y_{\inner} - \hat{y}_0} &\leq \norm{y_0 - \hat{y}_0}
  + B \norm{y_0 - \hat{y}_0}^{1+\eta}.
\end{align}
We now define the constant $\constsuper := \frac{BL}{\mu^{1 + \eta}}$ and assume that
$$
\dist(y_0, \cX_{\ast}) \leq \min\set{\left(\frac{\mu}{2\constb}\right)^{1/\eta},
      \left(\frac{\mu^{1-\eta}}{L \constsuper}\right)^{1 / \eta}}.
$$
Then $B \, \dist^{\eta}(y_0, \cX_{\ast}) \leq 1$ since
\[
    B \dist^{\eta}(y_0, \cM) \leq \frac{\mu^{1 - \eta}}{L \constsuper}
    \cdot \frac{\constsuper\mu^{1+\eta}}{L} = \frac{\mu^2}{L^2} \leq 1.
\]
Therefore by~\eqref{eq:manifolddecrease}, we have $\norm{y_{\inner} - \hat{y}_0} \leq 2 \norm{y_0 - \hat{y}_0}$. Consequently,
\begin{align*}
  \norm{y_{\inner} - y_0} \leq \norm{y_{\inner} - \hat y_0} +
  \norm{y_0 - \hat y_0} \leq 3\norm{y_0 - \hat y_0}
  \leq \frac{3}{\mu} f(y_0) \leq \tau f(y_0) .
\end{align*}
Thus, $\tilde x$ exists and $y_{\inner}$ satisfies
$
f(\tilde x) = \min_{y_j: \norm{y_j - y_0} \le \tau f(y_0)} f(y_j) \leq f(y_{\inner}).
$

Next we prove $f(\tilde x) \leq f(y_{\inner}) \leq \constsuper f(x)^{1+\eta}$.
To that end, note that $y_{\inner} \in B_{\delta}(\bar x)$. Indeed, since
$y_0 \in B_{\frac{\delta}{4}}(\bar x)$, we have
\[
  \norm{y_{\inner} - \bar{x}} \leq
  \norm{y_{\inner} - y_0} + \norm{y_0 - \bar{x}}
  \leq 4 \norm{y_0 - \bar{x}} < \delta,
\]
where the second inequality follows from $\norm{y_{\inner} - y_0} \leq
3 \norm{y_0 - \hat{y}_0}$ and the trivial bound $\norm{y_0 - \hat{y}_0}
\leq \norm{y_0 - \bar{x}}$.
Therefore, taking into account the Lipschitz continuity and sharpness of $f$ on
$B_{\delta}(\bar x)$, we find that
$$
f(\tilde x) \leq f(y_{\inner}) \leq L \|y_{\inner} - \hat y_{\inner} \| \leq BL \,
\dist^{1+\eta}(y_0, \cM) \leq \frac{BL}{\mu^{1+\eta}} f(y_0)^{1+\eta} =
  \constsuper f(x)^{1+\eta}.
$$
This completes the proof.

\subsection{Proof of Theorem~\ref{thm:SuperPolyak}}\label{sec:SuperPolyak}
We assume that \ref{item:assumption:main:sharpness},~\ref{item:assumption:main:b}, and~\ref{item:assumption:main:fallback} are in force in this section.
In the forthcoming proofs, we assume without loss of generality that
$f^{\ast} = 0$. We first show that the iterates $\set{x_{\outind}}_{k}$
exist and stay in a neighborhood of $\bar{x}$, so that every call to $\mathtt{FallbackAlg}$
produces a linearly convergent set of iterates; in turn, this shows that each iteration
of~\cref{alg:bundle-newton-method} must terminate.
\begin{lemma}
  \label[lemma]{lemma:iterates-remain-near-xast}
  The iterates $\set{x_{\outind}}_{\outind}$ exist and for all $\outind \geq 0$, satisfy
    \begin{equation*}
      \begin{aligned}
        f(x_{\outind+1})  &\leq \frac{1}{2} f(x_{\outind}); \\
       \norm{x_{\outind} - \bar{x}} &\leq \frac{1 - \rho}{2} \cdot
       \min\set{\frac{\delta}{4}, \radius_1}; \\
       \dist(x_{\outind}, \cX_{\ast}) &\leq \radius_2.
      \end{aligned}%\label{eq:iteratesexist}
    \end{equation*}
  \end{lemma}
\begin{proof}
  We begin with some some notation. Define the following four constants
  \[
    \omega := \frac{1 + \rho}{\mu(1 - \rho)}; \quad
    \delta_1 := \min\set{\frac{\delta}{4}, \radius_1}; \quad
    \delta_2 := \frac{\radius_2}{1 + \max\set{2 \kappa \, \frac{1 + \rho}{1 - \rho}, \frac{4L}{3}}};
  \]
  and
  $$
  c := \left[ \left(1 + \max\set{2 \kappa \, \frac{1 + \rho}{1 - \rho}, \frac{4L}{3}}\right)
  \left( \frac{2}{1 - \rho} \right)\right]^{-1}.
  $$
  In particular, we have $\norm{x_0 - \bar{x}} \leq c\delta_1$.
  Let $z_{j,k}$ denote the $j^{\text{th}}$ iterate of the
  call to $\mathtt{FallbackAlg}(x_{\outind}, \cA, \frac{1}{2}f(x_{\outind}))$. In addition, for all $j$ and $k$, let $\hat z_{j, k} \in
  P_{\cX_\ast}(z_{j,k})$ and $\hat x_{\outind} \in P_{\cX_\ast}(x_{\outind})$. Now we turn to the proof.

  If the iterates exist, clearly the inequality $f(x_{\outind+1})  \leq \frac{1}{2} f(x_{\outind})$ holds for all $\outind \geq 0$. Thus, we focus
  on the latter two bounds. In particular, define $x_{-1}:= x_0$. Then we claim the following three bounds hold for all $k \ge 0$:
  \begin{subequations}
  \begin{align}
    \norm{x_{\outind} - x_{\outind-1}} &\leq c_{\outind} f(x_0); \label{eq:distance-succesive} \\
    \norm{x_{\outind} - \bar{x}} &\leq
    c \delta_1 + f(x_0)  \sum_{j=0}^{\outind-1} c_j; \label{eq:distance-to-xbar} \\
    \dist(x_{\outind}, \cX_\ast) &\leq \delta_2 + f(x_0)  \sum_{j=0}^{\outind-1}
   c_j.
    \label{eq:dist-to-manifold}
  \end{align}
  \end{subequations}
  where we define $c_{\outind}:= \max\set{\omega \left(\frac{1}{2}\right)^{\outind},
  \left(\frac{3}{4}\right)^{\outind}}$ for all $\outind \geq 0$. We prove the claim by induction.

  The base case is satisfied by assumption.
  Now assume that~\eqref{eq:distance-succesive},~\eqref{eq:distance-to-xbar}, and~\eqref{eq:dist-to-manifold} hold up to index $\outind$.
  We first prove~\eqref{eq:distance-succesive}. To that end, first suppose that we have $x_{\outind+1} = \bundle(x_{\outind}, (3/2)^{\outind})$. Then
  \begin{align}
    \norm{x_{\outind+1} - x_{\outind}} &\leq \left( \frac{3}{2} \right)^{\outind} f(x_{\outind}) \leq
    \left( \frac{3}{4} \right)^{\outind} f(x_0) \leq c_{\outind} f(x_0),
    \label{eq:diff-bundle}
  \end{align}
  where the first inequality follows from the definition of
  \cref{alg:build-bundle-method} and the second inequality follows from
  the inductive hypothesis. Thus, in this case, \eqref{eq:distance-succesive} holds.
  On the other hand, suppose that $x_{\outind+1} \neq \bundle(x_{\outind}, (3/2)^{\outind})$.
  To show that the fallback method initialized at $x_{\outind}$ must terminate, we first bound
  $\|x_{\outind} - \bar x\|$ and $\dist(x_{\outind}, \cX_\ast)$. To that end, the inductive hypothesis
  ensures
  \begin{align}
    \norm{x_{\outind} - \bar{x}} \leq c \delta_1
    + f(x_0) \max\set{2 \omega, \frac{4}{3}}
                         \leq c \delta_1 \left(1
                         + \max\set{2\kappa \cdot \frac{1 + \rho}{1 - \rho}, \frac{4L}{3}}\right)
                          \leq
     \frac{1-\rho}{2} \delta_1,
    \label{eq:induction-cond-1}
  \end{align}
  where the second inequality follows from the Lipschitz continuity of $f$
  and the third inequality follows by definition of $c$.
  A similar argument yields
  \begin{equation}
    \dist(x_{\outind}, \cX_\ast) \leq
    \delta_2 + f(x_0) \max\set{\frac{2 (1 + \rho)}{\mu(1 - \rho)}, \frac{4}{3}}
    \leq \radius_2.
    \label{eq:induction-cond-2}
  \end{equation}
  Then by \eqref{eq:induction-cond-1},~\eqref{eq:induction-cond-2} and~\cref{lemma:fallback-stay-in-ball},
  the fallback algorithm must terminate by some iteration $\inner \in \NN$.
  In addition, $x_{\outind+1} = z_{\inner,\outind}$ and $x_{\outind} = z_{0,\outind}$. Consequently,
  \begin{align*}
    \norm{x_{\outind+1} - x_{\outind}} = \norm{z_{\inner,\outind} - z_{0,k}}
                         \leq
    \sum_{j=0}^{\inner-1} \norm{z_{j+1, \outind} - z_{j,\outind}}
                         &\leq
    (1 + \rho) \sum_{j=0}^{\inner-1} \norm{z_{j,\outind} - \hat{z}_{j,\outind}}
  \end{align*}
  where the second inequality follows from the triangle inequality and the following bound $\norm{z_{j+1, \outind} - \hat z_{j,\outind}} \leq \rho\norm{z_{j,\outind} - \hat{z}_{j,\outind}} $.
  Next, by \cref{lemma:fallback-stay-in-ball}, we have $\norm{z_{j, \outind} - \hat{z}_{j,\outind}} \leq \rho^j\norm{z_{0,k} - \hat{z}_{0,k}}$ for all $j< i$. Therefore,
\begin{align*}
   \norm{x_{\outind+1} - x_{\outind}} \leq   (1 + \rho)\sum_{j=0}^{\inner-1} \rho^j\norm{z_{0,k} - \hat{z}_{0,k}}
                         \leq
    \frac{1 + \rho}{1 - \rho} \norm{z_{0,\outind} - \hat{z}_{0,\outind}}
                         \leq
    \frac{1 + \rho}{\mu(1 - \rho)} f(x_{\outind}) \leq  c_{\outind} f(x_0),
\end{align*}
  where the
  fourth inequality follows from sharpness and the inclusion $z_{0,\outind} \in B_{\frac{\delta}{4}}(\bar{x})$.
  This proves~\eqref{eq:distance-succesive}.

  Next, we prove~\eqref{eq:distance-to-xbar}.
  To that end, observe that
  \begin{align*}
    \norm{x_{\outind+1} - \bar{x}} \leq \norm{x_{\outind} - \bar{x}} + \norm{x_{\outind+1} - x_{\outind}} &\leq c \delta_1 + f(x_0) \sum_{j=0}^{\outind-1} c_j + c_{\outind} f(x_0)
    = c \delta_1 + f(x_0) \cdot \sum_{j=0}^{\outind} c_j,
  \end{align*}
  where the second inequality follows by the inductive assumption and~\eqref{eq:distance-succesive}. This proves~\eqref{eq:distance-to-xbar}.

  Finally we prove~\eqref{eq:dist-to-manifold}. To that end, observe that
  \begin{align*}
    \dist(x_{\outind+1}, \cX_{\ast}) \leq \norm{x_{\outind+1} - \hat{x}_{\outind}}
    \leq \dist(x_{\outind}, \cX_{\ast}) + \norm{x_{\outind+1} - x_{\outind}}
    &\leq \delta_2 + f(x_0) \cdot \sum_{j=0}^{\outind-1}
    c_j + c_{\outind} f(x_0)  \\
    &\leq  \delta_2 + f(x_0) \cdot \sum_{j=0}^{\outind} c_j,
  \end{align*}
  where the third inequality follows by the inductive assumption and~\eqref{eq:distance-succesive}.
  This proves~\eqref{eq:dist-to-manifold} and completes the proof.
\end{proof}
An immediate corollary of~\cref{lemma:iterates-remain-near-xast,lemma:fallback-stay-in-ball} is the following:
\begin{corollary}
  \label[corollary]{lemma:polyak-sgm-oracle-complexity}
Every call to algorithm
to $\mathtt{FallbackAlg}$ in \cref{alg:bundle-newton-method} will terminate after
at most $\ceil{\frac{1}{1 - \rho} \, \log(2\kappa)}$ evaluations of $\cA$.
\end{corollary}
Finally, we show that all bundle steps are successful when $\outind \geq K_1$.
\begin{lemma}
  \label[lemma]{lemma:eventual-superlinear-improvement}
For all $i \geq 0$, we have the following:
\begin{enumerate}
\item \label{thm:SuperPolyak:item:superlinear:accept} $x_{K_1+i} = \bundle(x_{K_1+i-1}, (3/2)^{K_1 + i})$;
\item \label{thm:SuperPolyak:item:superlinear:sup} $
f(x_{K_1 + i}) - f^{\ast} \leq 2^{-\left(1 + \eta\right)^i}.
$
\end{enumerate}
\end{lemma}
\begin{proof}
  We begin with~\cref{thm:SuperPolyak:item:superlinear:accept}.
  To that end, we first show
  that for all $\outind \ge K_1$, the iterate $x_{\outind}$ and scalar $\tau = (3/2)^{\outind}$
  satisfy the assumptions of~\cref{corollary:function-value-reduction}.
  In particular, the vector $\tilde x$ exists and achieves the superlinear improvement~\eqref{eq:bigcprime}.
  Indeed, \cref{lemma:iterates-remain-near-xast} shows that
  $\norm{x_{\outind} - \bar{x}} \leq \frac{\delta}{4}$ for any $\outind$. Furthermore,
  by the sharp growth condition~\ref{item:assumption:main:sharpness} and the definition of $K_1$, we have
  \begin{align*}
    \dist(x_{\outind}, \cX_{\ast}) \leq \frac{f(x_{\outind})}{\mu} \leq
    2^{-K_1} \frac{f(x_0)}{\mu} &\leq
    \frac{1}{f(x_0) \max\set{\frac{(2\constb)^{1/\eta}}{\mu^{1 + 1/\eta}},
    \left(\frac{L\constsuper}{\mu}\right)^{1/\eta}}} \frac{f(x_0)}{\mu} \\
                           &=
    \min\set{\left(\frac{\mu}{2\constb}\right)^{1/\eta},
             \left(\frac{\mu^{1 - \eta}}{L \constsuper}\right)^{1/\eta}}.
  \end{align*}
  Finally, notice that $\left(\frac{3}{2}\right)^{\outind} > \frac{3}{\mu}$ for all $\outind
  \geq K_1$. Consequently, all the conditions of~\cref{corollary:function-value-reduction}
  are satisfied and thus the point
  \(
  \tilde{x}
  \)
  exists and satisfies
  \begin{align*}
    f(\tilde{x}) \leq \constsuper f(x_{\outind})^{1+\eta} \leq
    \constsuper \left( 2^{-K_1} f(x_0)\right)^{\eta} f(x_{\outind})
                 &\leq
                 \constsuper \left( 2^{-\log \left( f(x_0) \, (2\constsuper)^{1 / \eta} \right)} f(x_0)\right)^{\eta} f(x_{\outind}) \\
                 &\leq \frac{1}{2} f(x_{\outind}),
  \end{align*}
  where the second inequality follows from \cref{lemma:iterates-remain-near-xast}. This completes the proof
  of~\cref{thm:SuperPolyak:item:superlinear:accept}.

  We now prove~\cref{thm:SuperPolyak:item:superlinear:sup}.
  Define a sequence $\{a_k\}_k$ by $a_{\outind} := f(x_{\outind})$ for all $k \geq 0$.
  From~\cref{lemma:iterates-remain-near-xast,thm:SuperPolyak:item:superlinear:accept},
   \begin{equation}
    a_{\outind+1} \leq \begin{cases}
      \frac{1}{2} a_{\outind}, & \text{ if $\outind < K_1$;} \\
      \constsuper a_{\outind}^{1+\eta}, & \text{ otherwise}.
    \end{cases}
    \label{eq:f-decrease-allk}
  \end{equation}
  In particular, by definition of $K_1$, we have
  $$
  a_{K_1} \leq 2^{-K_1} a_0 \leq \frac{1}{2(\constsuper)^{1/\eta}}.
  $$
  Thus, unfolding~\eqref{eq:f-decrease-allk} shows that for all $i \geq 0$, we have
  \begin{align*}
    a_{K_1 + i} \leq \constsuper a_{K_1 + i - 1}^{1 + \eta} \leq
    \left(\constsuper\right)^{\sum_{j=0}^{i-1} (1 + \eta)^j} a_{K_1}^{(1 + \eta)^{i}}
    \leq \left((\constsuper)^{1/\eta} a_{K_1}\right)^{(1 + \eta)^i}
                   &\leq \left(\frac{1}{2}\right)^{(1 + \eta)^i}.
  \end{align*}
  This completes the proof of~\cref{thm:SuperPolyak:item:superlinear:sup}.
\end{proof}

To finish the proof, we tabulate the total number of evaluations of $\subg$ and $\cA$.
To that end, note that the first $K_1$ iterations each require at most $d$ evaluations of $\subg$ and $\ceil{\frac{1}{1 - \rho} \, \log(2\kappa)}$ evaluations of $\cA$ by the definition of $\bundle$ and Corollary~\ref{lemma:polyak-sgm-oracle-complexity}, respectively.
Each remaining step of the algorithm simply calls $\bundle$, which requires $d$ evaluations of $\subg$.
Therefore, since $f(x_{K_1 + i}) \leq \epsilon$ whenever
$$
  i \geq  \frac{\log \log_2 \left(\frac{1}{\epsilon} \right)}{
    \log(1+\eta)},
$$
the algorithm requires at most
$$
d K_1
+ d \ceil{\frac{\log \log_2 \left(\frac{1}{\epsilon} \right)}{\log(1+\eta)}} ,
$$
evaluations of $\subg$. In addition, since $\cA$ is only called during the first $K_1$ iterations, we evaluate $\cA$ at most $\ceil{\frac{1}{1 - \rho} \, \log(2\kappa)}K_1$ times. This completes the proof.

\subsection{Proof of Corollary~\ref{cor:corsuperpolyak}}\label{sec:corsuperpolyak}

The result is an immediate corollary of Lemma~\ref{lemma:one-step-improvement} and the Theorem~\ref{thm:SuperPolyak}

\subsection{Proof of Corollary~\ref{cor:newtonsemismooth}}\label{sec:cor:newtonsemismooth}

Let us first assume that $\cX_\ast$ is isolated at $\bar x$ and $F$ is semialgebraic.
In this case, Item~\ref{cor:newtonsemismooth:b} follows from Lemma~\ref{lem:basicexamples}.
Next we show that $f$ and $g$ satisfy Assumption~\ref{item:assumption:main:b} (the other assumption is immediate). Indeed, this immediately follows from Corollary~\ref{cor:semib} since the norm is semialgebraic.

\subsection{Proof of Corollary~\ref{cor:feasibilityfinal}}\label{sec:cor:feasibilityfinal}

In either case (i) or (ii), Item~\ref{lem:bregdistance} follows from Lemma~\ref{lem:bregdistance}.
Next we show that $f$ and $g$ satisfy Assumption~\ref{item:assumption:main:b} (the other assumption is immediate). Indeed, by Lemma~\ref{lem:bregdistance}, each pair $(\dist(\cdot, \cX_i), g_{\cX_i}^\T)$ is $(b)$-regular along $\cX_i$ at $\bar x$ with exponent $1+\eta$. Consequently, by Corollary~\ref{cor:bregsum} the pair $(f, g^\T)$ is $(b)$-regular along $\cX_\ast$ at $\bar x$ with exponent $1+\eta$, as desired.

\subsection{Proof of Corollary~\ref{cor:manifoldsuper}}\label{sec:cor:manifoldsuper}
Item~\ref{cor:feasibilityfinal:regular} is classical and shown for example in~\cite{adrianaltproj}.
Item~\ref{cor:feasibilityfinal:semismoothness} follows from Lemma~\ref{lem:bregdistance}.

\section{Numerical study and implementation strategies}\label{sec:experiments}

In this section, we present implementation strategies for the $\superpolyak$ algorithm and a brief numerical illustration. We begin with several implementation strategies.

\subsection{Implementation strategies}\label{sec:implementation}

In this section, we discuss  several strategies that we found to improve the numerical performance of $\superpolyak$.
\subsubsection{Early termination of~$\bundle$.}\label{sec:earlyterm}

We suggest terminating $\bundle$ early, returning some iterate $y_i$ whenever at least one of the following holds.
\begin{enumerate}
  \item ($\textbf{Rank deficiency}$) Suppose that there exists ${\inner} \leq d$ such that $\rank(A_j) = j$ for all $j < \inner$,
    but $\rank(A_{\inner}) < \inner$.
   In this case, we suggest that $\bundle$ return
   $$
   \tilde x =  \argmin_{y_j \colon j \leq \inner \text{ and } \norm{y_j - y_0} \le \tau f(y_0)} f(y_{j}).
   $$
  \item \label{remark:earlyterm:large}($\textbf{Large distance traveled}$) Suppose that there exists $\inner < d$ such that
  $\|y_{j} - y_0\| \leq \tau f(y_0)$ for all $j < \inner$, but $\|y_{\inner} - y_0\| > \tau f(y_0)$.
   In this case, we suggest that $\bundle$ return
   $$
   \tilde x =  \argmin_{y_j \colon j \leq \inner-1 \text{ and } \norm{y_j - y_0} \le \tau f(y_0)} f(y_{j}).
   $$
  \item \label{remark:earlyterm:super}($\textbf{Superlinear improvement}$)
  Suppose that we have an estimate $\eta_{\mathsf{est}}$ of the $(b)$-regularity exponent $\eta$ and that $f(y_0) - f^{\ast} < 1$. Suppose that we find an iterate
  $y_i$ such that
    \begin{equation}
      f(y_i) - f^{\ast} \leq \left(f(y_0) - f^{\ast}\right)^{1 + \eta_{\mathsf{est}}}.
      \label{eq:superlinear-improvement}
    \end{equation}
In this case, we suggest that $\bundle$ return the first such iterate $\tilde x = y_i$.
\end{enumerate}
 It is possible to show that these strategies still result in superlinear improvement, but we do not pursue this result here.

\subsubsection{Updating $\eta_{\mathsf{est}}$ in $\superpolyak$}

Item~\ref{remark:earlyterm:super} in Section~\ref{sec:earlyterm} may never be triggered if $\eta_{\est}$ is too large. We suggest the following simple update strategy.
Suppose that~\eqref{eq:superlinear-improvement} fails for all candidates $y_i$
but $x_{\outind+1} = \mathtt{PolyakBundle}(x_{\outind}, (3/2)^{\outind})$.
Then we update
\[
  \eta_{\mathsf{est}} = \max\left(\eta_{\mathsf{lb}}, q \cdot \eta_{\mathsf{est}}\right),
  \quad \text{where} \quad \eta_{\mathsf{lb}} \geq 0, \; q \in (0, 1).
\]
In our implementation, we set $\eta_{\mathsf{est}} = 1$,
$\eta_{\mathsf{lb}} = 0.1$, and $q = 0.9$.
It is straightforward to show that the iterates $x_k$ continue converge to superlinearly with this estimation strategy.

\subsubsection{Less frequent calls to $\bundle$}\label{sec:omegagamma}
In early iterations of  Algorithm~\ref{alg:bundle-newton-method}, the $\bundle$ procedure may not succeed. To reduce wasted computation, one may
\begin{enumerate}
\item Replace the factor $(3/2)^k$ in line~\ref{alg:threehalf} with $\omega^k$ for some $
\omega > 1$;
\item Replace the factor  $1/2$ in line~\ref{alg:fallbackstep} with some $0 < \gamma< 1$.
\end{enumerate}
In particular, it is straightforward to show that the iterates $x_k$ continue to converge superlinearly whenever $\omega \gamma < 1$.
These changes have separate effects.
Reducing $\omega$ makes the early termination strategy in Item~\ref{remark:earlyterm:large} of the strategies outlined in Section~\ref{sec:earlyterm} more likely to be triggered, lowering the cost of early unsuccessful $\bundle$ steps.
On the other hand, reducing $\gamma$ results in less frequent calls to $\bundle$.
%}

\subsubsection{Computing $\bundle$ in $O(d^3)$ operations}\label{sec:fasterbundle}

Each step of $\bundle$ requires the evaluation of $A_{\inner}^{\dag} w$ for some $w
\in \RR^{\inner}$.
When $\inner \leq d$, the pseudoinverse of an $\inner \times d$ matrix may be computed explicitly
in $O(d\inner^2)$ floating point operations (flops)~\cite{GVL13}.
Consequently, using this method, one may compute all matrices $A_{\inner}^{\dag}$ (for $\inner = 1, \ldots, d$) with $O(d^4)$ flops.
In this section, we point out that there exists a more efficient method for ``updating" the pseudoinverse $A_{\inner}^{\dag}$.
Each update requires $O(d^2)$ flops, bringing the total cost down to $O(d^3)$ flops. The method is based on iteratively updating
the QR decomposition of $A_{\inner}^{\T}$, which allows efficient evaluation of $A_{\inner}^\dag w$ for arbitrary $w \in
\RR^{\inner}$. We place the proof and description of the algorithm in Appendix~\ref{appendix:bundleefficient}.

\begin{proposition}\label{prop:fastbundle}
Consider the setting of Proposition~\ref{corollary:function-value-reduction}.
Then there exists a method to return the point $\tilde x = \bundle(x, \tau)$
using at most $O(d^3)$ flops (ignoring the cost of evaluations of $f$ and $\subg$).
\end{proposition}

\subsection{Numerical illustration}
We now briefly illustrate the numerical performance of $\superpolyak$ on several signal recovery applications, including low-rank matrix sensing, max-linear regression, phase retrieval, and compressed sensing.
In each experiment, we run $\superpolyak$ with a natural first-order fallback method, which we also use a baseline method for comparison against.
Our main finding is that $\superpolyak$ often improves -- both in oracle complexity and time -- on several first-order fallback methods, including the Polyak subgradient method, the method of alternating projections, and the classical fixed-point iteration.
Our code is available at the following URL: \url{https://github.com/COR-OPT/SuperPolyak.jl}.

\paragraph{Implementation Details.} Throughout we use the default scaling factors $\omega = \frac{3}{2}$ and
$\gamma = \frac{1}{2}$ described in Section~\ref{sec:omegagamma}.
We also use the algorithmic enhancements from~\cref{sec:implementation}.
In each experiment, we fix a minimizer $\bar x$ of $f$ and initialize both $\superpolyak$ and the fallback method at a uniform random point $x$ satisfying $\|\bar x -  x\|/\|\bar x\| = 1$; an exception is the basis pursuit experiment of Section~\ref{sec:compressedsensing}, which is initialized at the zero vector.
In each experiment, the problem data and initializer are chosen randomly; we found that the depicted behavior was similar across multiple runs of the algorithm, so we plot only one instance in each figure.
With the exception of the experiment in Figure~\ref{fig:conditioning-hadamard}, all generalized gradients $g$ were computed via the automatic differentiation library $\mathtt{ReverseDiff.jl}.$
The experiments in Figures~\ref{fig:signal-recovery-problems},~\ref{fig:max-linear-regression},~\ref{fig:complex-phase-retrieval},
and~\ref{fig:lasso-regression} were performed on an Intel Core i7-7700 CPU desktop with 16GB of RAM running Manjaro Linux.
The experiment in Figure~\ref{fig:conditioning-hadamard} was performed on a shared Intel Xeon E5-2680 (v3) cluster with a 16GB RAM limit
running Ubuntu Linux. We used Julia v1.6.1 in both environments.

\subsubsection{Low-rank matrix-sensing and $\polyak$}
\label{sec:matrix-sensing}

In this problem, we observe a measurement vector $\bar  y \in \Rbb^m$ satisfying
\begin{equation*}
 \bar  y = \cA(\bar  M) + \bar \xi,
\end{equation*}
where $\bar M \in \RR^{d \times d}$ is a fixed rank $r$ matrix, $\cA: \Rbb^{d \times d} \to \Rbb^m$ is a linear operator and
$\bar \xi \in \Rbb^m$ is a ``noise" vector.
The goal of the low-rank matrix sensing problem is to recover $\bar M$.
Recoverability depends on the operator $\cA$.
In this section, we consider linear operators $\cA$ with rows $\cA_i$, satisfying
$$
\cA_i(M) = \dotp{\ell_i, M r_i} \qquad  \text{for some $\ell_i, r_i \in \RR^d$ and all $M \in \RR^{d\times d}$}.
$$
The work~\cite{CCD+21} analyzes the following objective for this problem class:
\begin{align}\label{eq:lowrankobjective}
 f(U, V) := \frac{1}{m} \norm{\cA(UV^{\T}) - y}_1 \qquad \text{for all $U, V \in \RR^{d \times r}$},
\end{align}
In particular, \cite{CCD+21} shows that $f$ satisfies~\ref{item:assumption:main:sharpness} when (i) $\ell_i, r_i$ are i.i.d.\ standard Gaussian vectors, (ii) $m \gtrsim rd$,  and (iii) at most a small constant fraction of entries of $\bar \xi$ are nonzero.
Moreover, any solution $(\bar U, \bar V)$ satisfies $\bar M = \bar U \bar V^\T$. Thus, Proposition~\ref{prop:basicsemialgebraic} implies that $f$ satisfies~\ref{item:assumption:main:b}.

We perform experiments with $\superpolyak$ using $\polyak$ as the fallback method in two different settings. In both settings, we set $\bar \xi = 0$, which leads to optimal value $f^\ast = 0$. Note that even in this setting, the nonsmooth $\ell_1$ is penalty is preferable to $\ell_2$ since it leads to better conditioning~\cite{CCD+21}. Note also that the total number of parameters we optimize over is $2dr$.
\begin{enumerate}
\item {\bf (Varying dimensions/ranks)} In this setting, we choose $\bar M = \bar U\bar V^\T$ where $\bar U, \bar V \in \RR^{d \times r}$ are uniform random $d\times r$ matrices with orthonormal columns.
In addition, we choose $\ell_i, r_i$ to be i.i.d.\ standard Gaussian vectors.
Figure~\ref{fig:signal-recovery-problems} then compares $\superpolyak$ to $\polyak$ to for varying $(d, r)$ and $m = 3rd$.
\item {\bf (Effect of conditioning of $\bar M$)} In this setting, we choose $\bar M = \bar U\Lambda\bar V^\T$ where $\Lambda \in \RR^{r\times r}$ is a diagonal matrix with condition number $\tilde \kappa$ and $\bar U, \bar V \in \RR^{d \times r}$ are uniform random $d\times r$ matrices with orthonormal columns. We then compare  $\superpolyak$ to $\polyak$ to for varying $\tilde \kappa$ under two measurement models:
\begin{enumerate}
\item {\bf (Gaussian measurements/small dimension)} Figure~\ref{fig:conditioning-gaussian} chooses $\ell_i, r_i$ to be i.i.d.\ standard Gaussian vectors. Here, $2dr = 4000$.
\item {\bf (Hadamard measurements/medium dimension)} Figure~\ref{fig:conditioning-hadamard} presents a similar, but larger scale comparison using measurements $\ell_i$ and $r_i$ from a random Hadamard ensemble (which allows for faster matrix vector products); see~\cite[Section 6.3]{DR18} for a formal description of the measurements. Here, $2dr = $ 131,072.
\end{enumerate}
\end{enumerate}
In the experiments, $\superpolyak$ converges superlinearly and requires only a fraction of the oracle calls to $\subg$ compared $\polyak$.
With the exception of a low-dimensional instance in Figure~\ref{fig:conditioning-gaussian}, the phenomenon persists when comparing CPU times.

\begin{figure}[h!]
  \centering
  \includegraphics[width=0.95\textwidth]{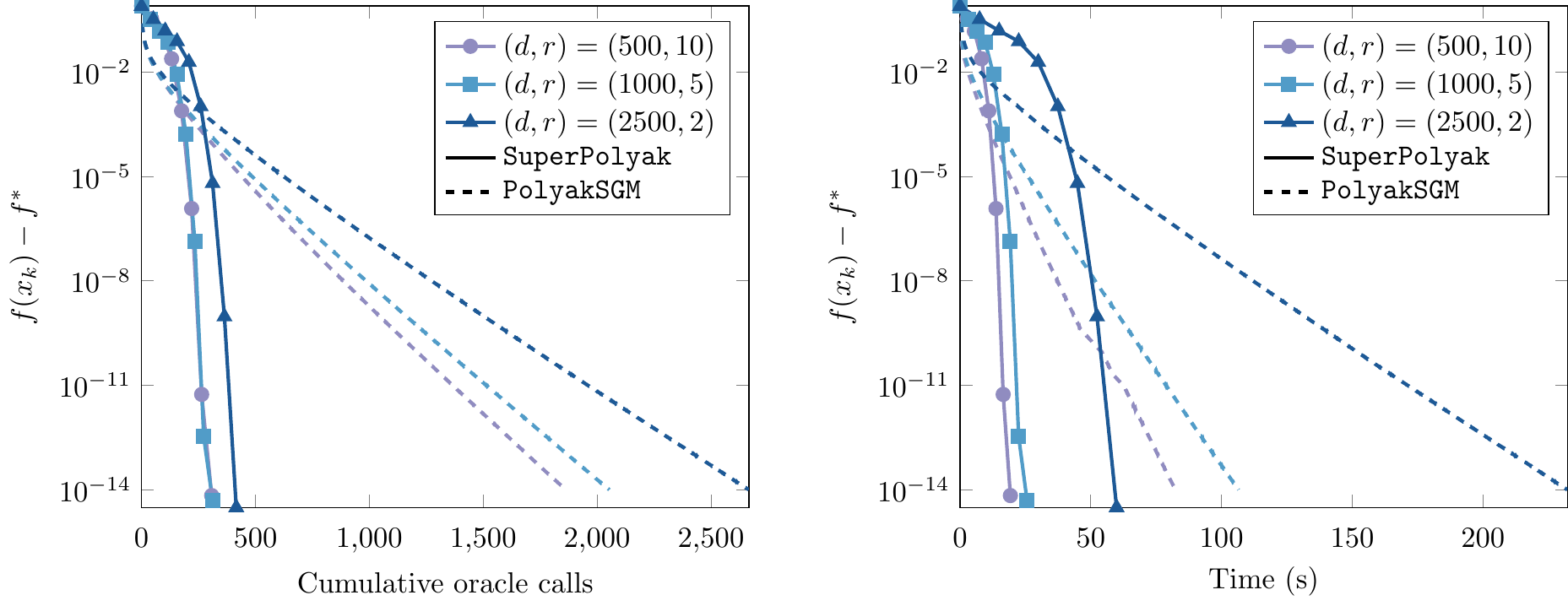}
  \caption{Low-rank matrix sensing with Gaussian measurements, varying dimension/ranks, and $m = 3rd$. See~\cref{sec:matrix-sensing} for description.}
  \label{fig:signal-recovery-problems}
\end{figure}
\begin{figure}[h]
  \centering
  \includegraphics[width=0.95\textwidth]{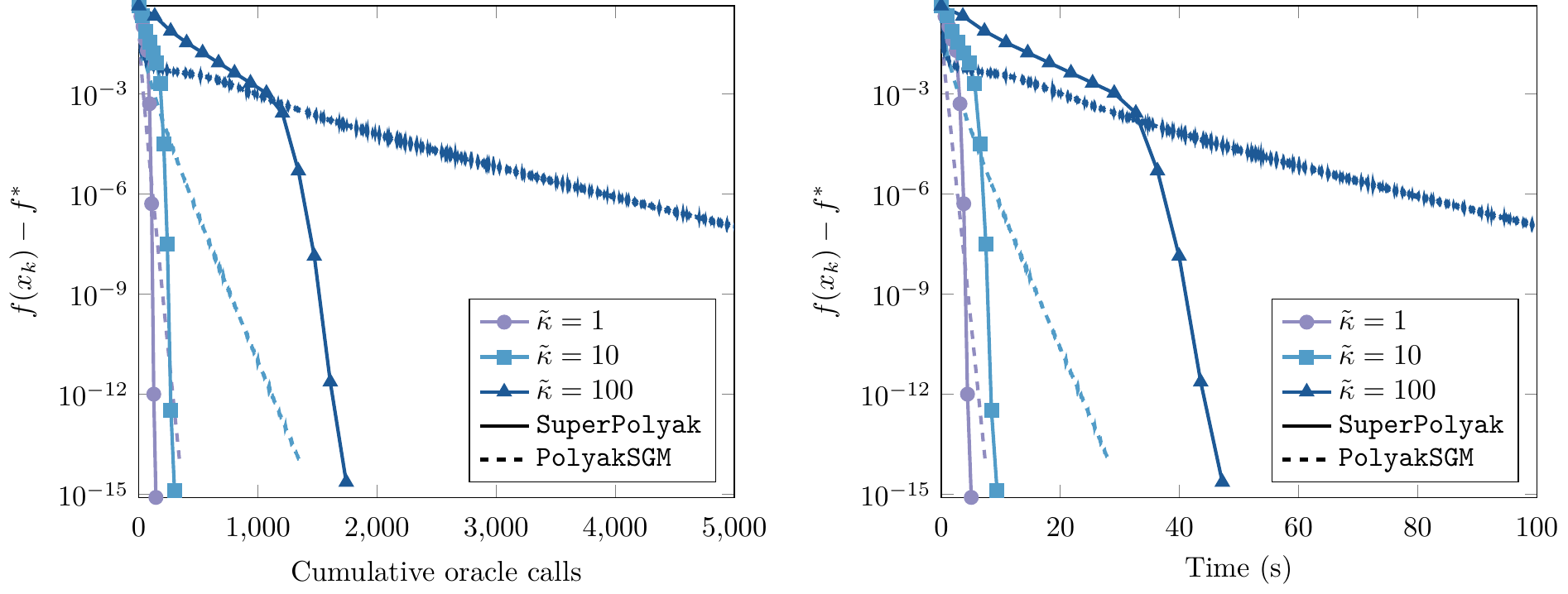}
  \caption{Low-rank matrix sensing with Gaussian measurements, varying condition number $\tilde \kappa$, and parameters $d = 500$, $r = 4$ and $m = 5rd$.
  See~\cref{sec:matrix-sensing} for description.}
  \label{fig:conditioning-gaussian}
\end{figure}

\subsubsection{Max-linear regression and $\polyak$}\label{sec:maxlinear}
In this problem, we observe a measurement vector $\bar y \in \RR^m$ satisfying
\begin{equation*}
 \bar  y_i = \max_{j \in [r]} \set{ \dotp{ \bar \beta_j, a_i} } \qquad \text{ for $i = 1, \ldots, m$},
\end{equation*}
for known standard Gaussian vectors $a_i$ ($i \in [m]$) and unknown vectors $  \bar\beta_j$ ($j \in [r]$).
This problem is an instance of the support function regression problem~\cite{PW90,Gun12}, where we observe several random evaluations of the support function of $\text{conv}\{ \bar\beta_1, \ldots,  \bar\beta_r\}$ and we seek to recover the vertices $ \bar\beta_j$.
To recover $ \bar\beta_1, \dots,  \bar\beta_r$,
we optimize the following objective
\begin{equation*}
 f(\beta_1, \ldots, \beta_r) :=
  \frac{1}{m} \sum_{i=1}^m \abs{y_i - \max_{j \in [r]} \set{\dotp{\beta_j, a_i}}} \qquad \text{ for all $\beta_1, \ldots, \beta_r \in \RR^d$.}
\end{equation*}
We do not attempt to verify assumptions~\ref{item:assumption:main:sharpness} and~\ref{item:assumption:main:b} for this problem class.
Instead, we note that if the solution set is isolated, semialgebraicity of $f$ implies that we~\ref{item:assumption:main:b} holds (see Proposition~\ref{prop:basicsemialgebraic}). On the other hand, verifying Assumption~\ref{item:assumption:main:sharpness} is remains an intriguing open problem.

We now turn to our experiment. For simplicity, we sample each $\bar \beta_j$ uniformly from $\sphere^{d-1}$. Then in Figure~\ref{fig:max-linear-regression}, we apply $\superpolyak$ with fallback method $\polyak$.
Again we see that $\superpolyak$ outperforms the $\polyak$ method and appears insensitive to the number of problem parameters $dr$.

\begin{figure}[h!]
  \centering
  \includegraphics[width=0.95\textwidth]{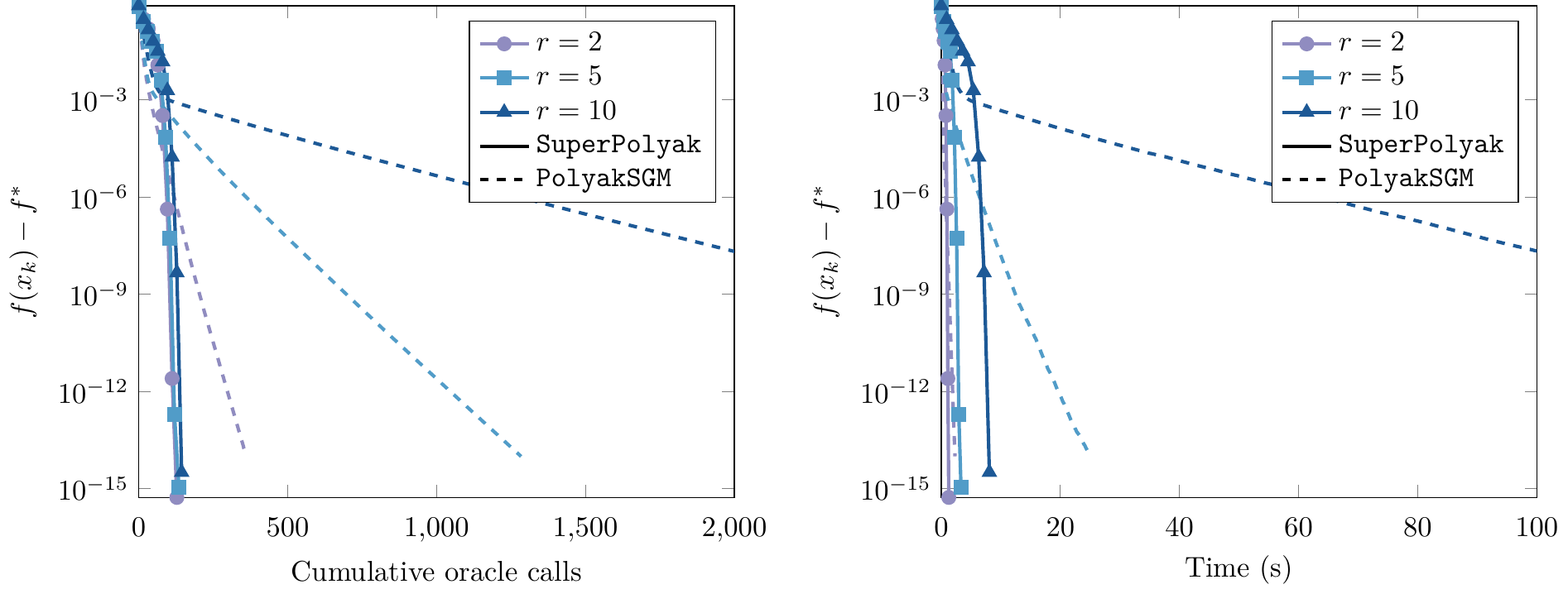}
  \caption{Max-linear regression with Gaussian measurements, varying $r$, and parameters $d = 500$ and $m = 3dr$. See~\cref{sec:maxlinear} for description.}
  \label{fig:max-linear-regression}
\end{figure}

\subsubsection{Phase retrieval and the method of alternating projections}
\label{sec:phase-retrieval-altproj}
In this problem, we observe a measurement vector $\bar{y} \in \RR^m$ satisfying
\[
  \bar{y}_i = \abs{\dotp{a_i, \bar{x}}} \qquad \text{for $i = 1, \dots, m$},
\]
for known measurement vectors $a_i$ $(i \in [m])$ and an unknown signal
$\bar{x} \in \CC^d$, with the goal of recovering $\bar{x}$ up to phase.%
\footnote{For this section, recall that $\dotp{x, y} = \trace(x^{\mathsf{H}} y)$ where $x^{\mathsf{H}}$ is
  the conjugate transpose of $x$ and $\abs{x} = \sqrt{\Re^2(x) + \Im^2(x)}$
for any $x$, $y \in \CC^d$.
As usual, we also identify $\CC^d$ with $\RR^{2d}$ in order to apply the results of this manuscript. To be consistent with the rest of the notation of the paper, we use $i$ as an index, not the imaginary unit.}
To recover $\bar x$, we consider the feasibility formulation:
\begin{equation}
  \mathrm{find} \quad \hat {y} \in \cY_1 \cap \cY_2 \qquad
  \text{where} \qquad
  \cY_1 := \set{
    u \in \CC^m \mid \abs{u} = y
  }; \;\; \cY_2 := \range(A);
  \label{eq:waldspurger-set-intersection}
\end{equation}
and $A \in \CC^{m \times d}$ is the matrix whose $i$th row is $a_i^{\mathsf{H}}$. Given $\hat y$ in the intersection, we then estimate $\bar{x}$ with $\hat{x} = A^{\dag} \hat{y}$.
Note that when $A$ is generic and $m \geq 4d - 4$, any such solution $\hat x$ is unique up to a global phase~\cite{BCE06,CEHV15}.
To solve this feasibility formulation, we consider the following objective:
\[
  f(y) = \dist(y, \cY_1) + \dist(y, \cY_2).
\]
Since $\cY_1$ and $\cY_2$ are smooth manifolds,  Corollary~\ref{cor:feasibilityfinal} shows that $f$ satisfies~\ref{item:assumption:main:b} at any point $\hat y \in \cY_1 \cap \cY_2$.
On the other hand, we were not able to locate property~\ref{item:assumption:main:sharpness} in the literature, even when the $a_i$ follow a complex Gaussian distribution. Nevertheless, there is reason to believe it holds in the Gaussian setting, since the method of alternating projections (described in~\ref{example:feasibility}) locally linearly converges to an element of $\cY_1 \cap \cY_2$~\cite{Waldspurger18}.

We now turn to our experiment. We generate $A$ with i.i.d.\ complex Gaussian entries
and sample $\bar{x}$ uniformly from the unit sphere in $\CC^d$, using $m = 4d$ measurements
for varying dimension $d$. In~\cref{fig:complex-phase-retrieval}, we apply $\superpolyak$
with the method of alternating projections (see Example~\ref{example:feasibility}) as the fallback method.
Here, the oracle complexity of $\superpolyak$ is the number of evaluations of $P_{\cY_1} \circ P_{\cY_2}$ plus the number of subgradient evaluations of $f$.
We see that $\superpolyak$ improves upon the method of alternating projections both in terms of oracle complexity and time.

\begin{figure}[h]
  \centering
  \includegraphics[width=0.95\textwidth]{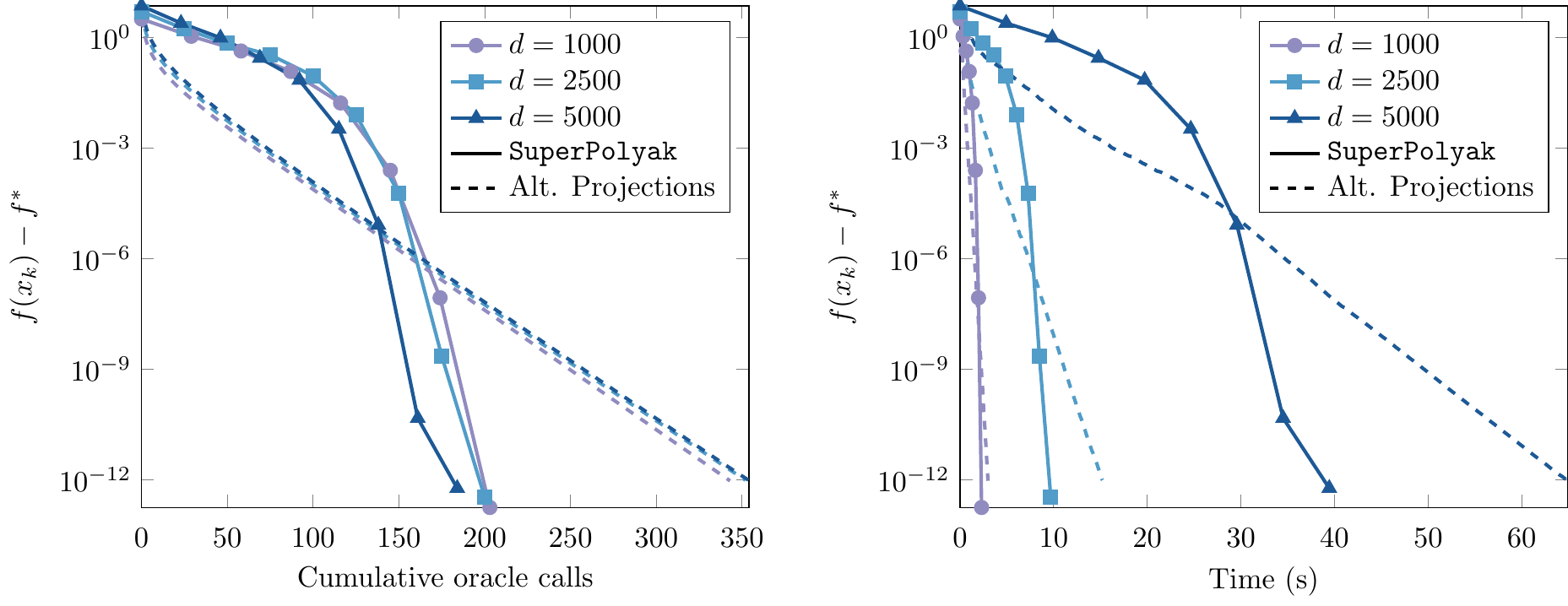}
  \caption{Complex phase retrieval with $m = 4d$ Gaussian measurements.
  See~\cref{sec:phase-retrieval-altproj} for description.}
  \label{fig:complex-phase-retrieval}
\end{figure}

\subsubsection{Compressed sensing and the proximal gradient method}\label{sec:compressedsensing}

In this problem, we observe a measurement vector $\bar y \in \RR^m$ satisfying
\begin{equation*}
  \bar y = A\bar x + \xi,%,
\end{equation*}
where $A \in \RR^{m\times d}$ is a known matrix, $\xi$ is an unknown noise vector, $\bar x$ is an unknown sparse vector.
The goal of the compressed sensing problem~\cite{Don06} is to recover $\bar x$ when $m$ is on the order of the number of nonzeros of $\bar x$.
There are several optimization based formulations for finding $\bar x$.
For simplicity we focus on the ``basis pursuit'' formulation~\cite{CDS01}, which solves
the following $\ell_1$-penalized least squares problem:
\begin{equation*}
 h(x) := \frac{1}{2}\norm{Ax - y}^2 + \lambda \norm{x}_1 \qquad \text{for all $x \in \RR^d$.}
  \label{eq:cs-lasso}
\end{equation*}
A standard approach for minimizing $h$ is the proximal gradient method, which iterates
\begin{align}
x_{\inner+1} := T(x_\inner),   \label{eq:prox-grad-step}
\end{align}
where for fixed $\tau > 0$ we define $T(x) := \prox_{\lambda \norm{\cdot}_1}\left(x - \tau A^{\T} (Ax - y)\right)$ for all $x \in \RR^d$. This motivates us to consider the following objective
$$
f(x) := \norm{x - Tx} \qquad \text{for all $x \in \RR^d$},
$$
which has the same minimizers as $h$ and has minimal value $f^\ast = 0$. This objective satisfies~\ref{item:assumption:main:sharpness} automatically and the fixed-point iteration~\eqref{eq:prox-grad-step} is a valid algorithmic mapping in the sense of~\ref{item:assumption:main:fallback}; see~\cite{tseng2010approximation}.
Moreover, when $A$ is drawn from a continuous distribution, the minimizer of $f$ is unique for any positive $\lambda$ with probability 1~\cite{Tibs13}.
Consequently, since $f$ is semialgebraic, Proposition~\ref{prop:rootfinding} shows that it satisfies~\ref{item:assumption:main:b} (see also Corollary~\ref{cor:semib} below).

We now turn to our experiment. Here, we choose $A$ with i.i.d.\ Gaussian entries and we vary dimension $d$, the number of nonzeros $s$ of $\bar x$, and the number of measurements $m$.
Then in Figure~\ref{fig:lasso-regression}, we apply $\superpolyak$ with fallback method~\eqref{eq:prox-grad-step}.
Here, the oracle complexity of $\superpolyak$ is the number of evaluations of $T$ plus the number of subgradient evaluations of $f$.
We find that $\superpolyak$ converges superlinearly and outperforms the fixed-point iteration~\eqref{eq:prox-grad-step} in both oracle evaluations and CPU time.

\begin{figure}[H]
  \centering
  \includegraphics[width=0.95\textwidth]{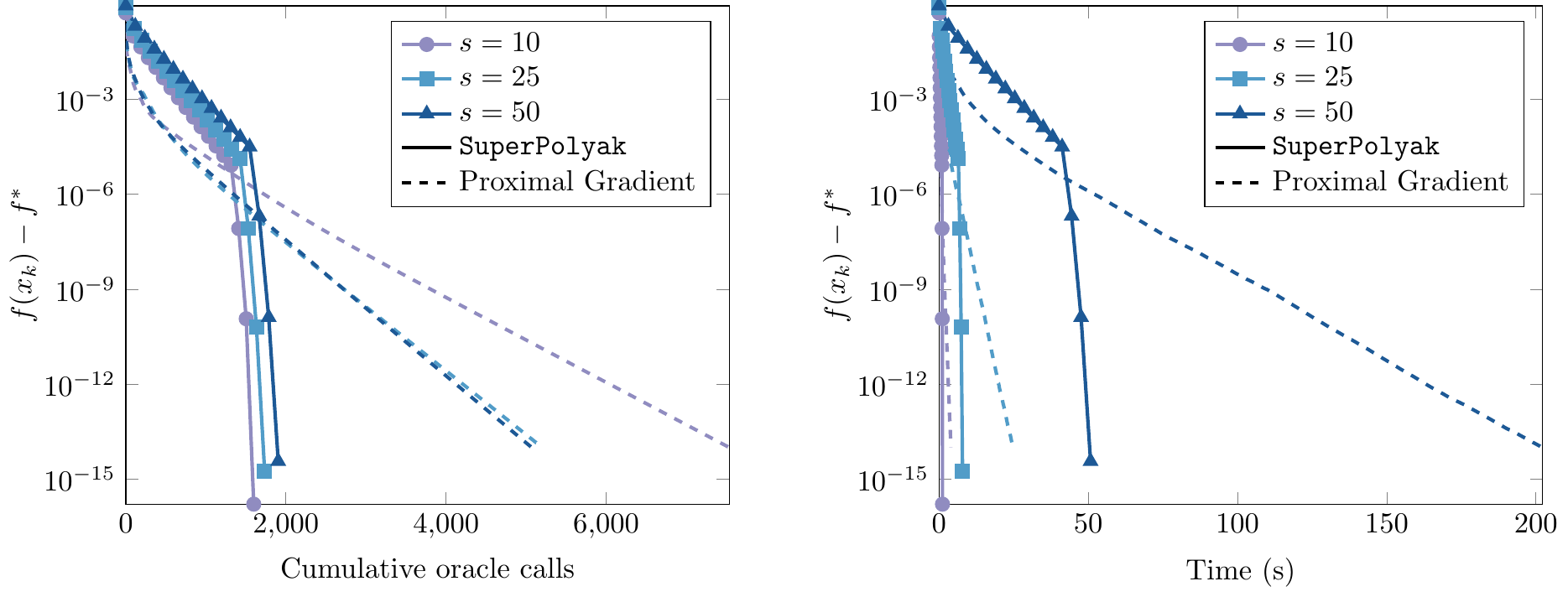}
  \caption{Basis pursuit with Gaussian measurements, varying sparsity $s$, and parameters $d = 100s$ and $m = 10s$.
  See~\cref{sec:compressedsensing} for description.}
  \label{fig:lasso-regression}
\end{figure}

\bibliographystyle{plain}
\bibliography{bibliography}

\appendix

\section{Proofs of auxiliary results}

\subsection{Proofs from Section~\ref{sec:experiments}}\label{appendix:experimentproofs}

\subsubsection{Proof of Proposition~\ref{prop:fastbundle}}\label{appendix:bundleefficient}

In this section, we briefly sketch how to compute the iterates $y_i$ of $\bundle$ by incrementally updating the ``reduced $QR$ decomposition" of $A^\T$. We begin with the following Lemma, which follows immediately from~\cite[Section 5.5.5]{GVL13}.
\begin{lemma}
  \label[lemma]{lemma:qr-apply-pseudoinverse}
  Consider $A \in \Rbb^{m \times n}$ with $m \leq n$ and $\rank(A) = m$. Then
  $A^\dag = QR^{-\T}$, where $A^{\T} = QR$ is the reduced QR decomposition of $A^{\T}$. Moreover, given any $b \in \RR^m$, the vector $A^\dag b$ can be computed in time $O(nm)$.
\end{lemma}

From~\cref{lemma:qr-apply-pseudoinverse}, it follows that computing
\[
  y_{\inner} = y_{0} - A_{\inner}^{\dag}\bmx{
  f(y_j) - f^{\ast} + \dotp{v_j, y_0 - y_j}}_{j=0}^{\inner}
\]
in~\cref{alg:build-bundle-method} is possible in time $O(d\inner)$ as long
as the QR decomposition of $A_{\inner}^{\T}$ is available and $A_{\inner}$ is
full row rank. Recalling the Lemma~\ref{lemma:manifold-alternatives}, we observe that if $A_\inner$ is rank deficient, we must have already obtained superlinear improvement. At that point, no further iterates $y_i$ need be computed (as suggested in~\cref{sec:earlyterm}). Thus, we now sketch how to efficiently maintain the QR decomposition of $A_{\inner}^{\T}$ while $A_i^\T$ is full rank:
\begin{enumerate}
  \item Initially, $A_1^{\T} = v_0$ and computing its QR factorization is trivial.
  \item At step $\inner$, we have $A_{\inner+1}^{\T} = \bmx{A_{\inner}^{\T} & v_{\inner}}$.
    where $A_{\inner}$ is full rank and its QR factorization is known. We consider the
    following cases:
    \begin{itemize}
      \item If $A_{\inner+1}$ remains full rank, we can compute its QR factorization
        with $O(d^2)$ flops using the algorithm from~\cite[Section 6.5.2]{GVL13}.
      \item If $A_{\inner+1}$ becomes rank-deficient, we discard the maintained
        QR decomposition, compute the product $A_{\inner+1}^{\dag} w$ explicitly
        using $O(d^3)$ flops, and exit the algorithm.
    \end{itemize}
\end{enumerate}
The above procedure requires $O(d^2)$ flops for every QR update step and $O(d^3)$ for
applying $A_{\inner}^{\dag}$ if $A_{\inner}$ becomes rank-deficient. The former can
happen at most $d$ times, while the latter clearly happens at most once. Therefore,
the total cost is $O(d^3)$ flops.

\vspace{10pt}

\noindent Finally, we briefly discuss the storage requirements of the above algorithm.
The incremental update algorithm of~\cite[Section 6.5.2]{GVL13} requires computing
the product $Q^{\T} v_i$, where $Q$ is the $d \times d$ orthogonal matrix from the
{full} QR factorization of $A_i^{\T}$. Implemented naively, this requires storing
$O(d^2)$ elements for $Q$. However, we can take advantage of the so-called
\textit{compact WYQ} format~\cite{SvL89} to decompose $Q$ as:
\[
  Q = I - U T U^{\T}, 
\]
for certain $U \in \RR^{d \times \inner}$ and upper triangular $T \in \RR^{\inner \times \inner}$. Given $U$ and $T$, we can compute
\(
Q^{\T} v_i = v_i - U^{\T} T^{\T} U v_i
\)
in $O(di)$ flops; moreover, the compact WYQ representation can be updated
in time $O(d^2)$ after adding a column to $A_i^{\T}$. Therefore, the algorithm
retains its computational complexity and requires storing at most $O(d\ell)$ numbers,
where $\ell$ is the maximal iteration index.

\subsection{Proofs from~\cref{proof:corollary:function-value-reduction}}
\label[subsection]{sec:build-bundle-missing-proofs}

\subsubsection{Proof of Proposition~\ref{proposition:build-bundle-sval-lower-bound}}
\label{sec:proof:proposition:build-bundle-sval-lower-bound}
The proof is a consequence of the following lemma.

\begin{lemma}
  \label[lemma]{lemma:progress-bundle-one-step-general}
  Consider a matrix $A \in \Rbb^{m \times d}$ and let
  $v \in \RR^{d}$ satisfy
  \[
    \norm{v} \leq L; \quad \norm{P_{\ker(A)}(v)} > \alpha > 0.
  \]
  Suppose that $\rank(A) = k$ for some $k \leq d$. Then the following holds:
  \begin{equation}
    \sigma_{k+1}\left(\bmx{A \\ v^{\T}}\right) \geq
    \frac{\alpha}{\sqrt{2}} \, \min\set{1, \left(\frac{\sigma_k(A)}{L}\right)}.
  \end{equation}
\end{lemma}
\begin{proof}
Define $w := \frac{P_{\ker(A)} v}{\norm{P_{\ker(A)} v}}$ and $\bar v = \frac{v}{\|v\|}$.
Observe that by the Davis-Kahan theorem~\cite{DK70}:
\begin{align}\label{eq:rank2distance}
  \opnorm{ww^{\T} - \bar v \bar v^{\T}} = \sqrt{1 - \dotp{w, \bar v}^2} \leq \sqrt{1 - \frac{\alpha^2}{\|v\|^2}}.
\end{align}
Consequently, we have the following
\begin{align*}
\sigma_{k+1}\left(\bmx{A \\ v^{\T}}\right) &=  \lambda_{k+1}(A^{\T} A + vv^{\T})\notag\\
&= \lambda_{k+1}(A^{\T} A + \norm{v}^2 \bar{v} \bar{v}^{\T})\notag\\
&\geq
  \lambda_{k+1}(A^{\T} A + \min\set{\sigma_k^2(A), \norm{v}^2} \bar{v} \bar{v}^{\T}) \notag \\
                                    &\geq
  \lambda_{k+1}(A^{\T} A + \min\set{\sigma_k^2(A), \norm{v}^2} ww^{\T}) -
  \min\set{\sigma_k^2(A), \norm{v}^2} \opnorm{ww^{\T} - \bar{v} \bar{v}^{\T}} \notag \\
                                    &=
  \min\set{\sigma_k^2(A), \norm{v}^2} -
  \min\set{\sigma_k^2(A), \norm{v}^2} \opnorm{ww^{\T} - \bar{v} \bar{v}^{\T}} \notag \\
                                    &\geq
  \min\set{\sigma_{k}^2(A), \norm{v}^2} \left(1 - \sqrt{1 - \frac{\alpha^2}{\norm{v}^2}}\right),
\end{align*}
where the first inequality follows since eigenvalues preserve the Loewner order, the second inequality
follows from Weyl's inequality, the third equality follows from the inclusion $w \in \ker(A)$ and
$\rank(A) = k$, and the third inequality follows from~\eqref{eq:rank2distance}.
Finally, applying $\sqrt{1 - x} \leq 1 - \frac{x}{2}$ to the lower bound above, we obtain
\begin{align*}
  \lambda_{k+1}(A^{\T} A + vv^{\T}) \geq
  \frac{\min\set{\sigma_k^2(A), \norm{v}^2} \alpha^2}{2 \norm{v}^2}
                                    \geq
    \min\set{\frac{\sigma_k^2(A) \alpha^2}{2L}, \frac{\alpha^2}{2}}
\end{align*}
as desired.
\end{proof}

\begin{proof}[Proof of~\cref{proposition:build-bundle-sval-lower-bound}]
  The proof follows by iterating~\cref{lemma:progress-bundle-one-step-general}
  for all $i \leq k$.
\end{proof}

\end{document}